\documentclass[11pt,reqno]{amsart}
\usepackage{mathrsfs}
\usepackage[american]{babel}
\usepackage{todonotes}
\usepackage{comment}
\usepackage{amsmath}
\usepackage{dsfont,mathtools,amssymb}
\usepackage{appendix}
\usepackage{nicematrix}
\usepackage[hidelinks]{hyperref}
\mathtoolsset{showonlyrefs,showmanualtags}
\usepackage[shortlabels]{enumitem}
\allowdisplaybreaks

\theoremstyle{theorem}
\newtheorem{theorem}{Theorem}[section]
\newtheorem{proposition}[theorem]{Proposition}
\newtheorem{lemma}[theorem]{Lemma}
\newtheorem{corollary}[theorem]{Corollary}
\theoremstyle{remark}
\newtheorem{remark}[theorem]{Remark}
\theoremstyle{definition}

\newtheorem{assumption}[theorem]{Assumption}

\newtheorem*{acknowledgement}{Acknowledgements}

\numberwithin{equation}{section}
\numberwithin{figure}{section}

\newcommand{\R}{\mathbb{R}}

\newcommand{\C}{\mathbb{C}}

\newcommand{\N}{\mathbb{N}}
\newcommand{\dd}{{\rm d}}
\newcommand{\fstop}{\; \text{.}}
\newcommand{\comma}{\; \text{,}\;\;}

\newcommand{\restr}[1]{\rvert_{#1}}

\newcommand{\tonde}[1]{\left(#1\right)}
\newcommand{\quadre}[1]{\left[#1\right]}

\newcommand{\tttonde}[1]{(#1)}
\newcommand{\abs}[1]{\left\lvert#1\right\rvert}
\newcommand{\tabs}[1]{\big\lvert#1\big\rvert}
\newcommand{\ttabs}[1]{\lvert#1\rvert}

\newcommand{\emparg}{{\,\cdot\,}}

\newcommand{\eqdef}{\coloneqq}
\newcommand{\defeq}{\eqqcolon}
\newcommand{\car}{\mathds{1}}
\newcommand{\scalar}[2]{\left\langle #1\, \middle \vert\, #2 \right\rangle}

\newcommand{\ttscalar}[2]{\langle #1 | #2 \rangle}
\newcommand{\norm}[1]{\left\lVert#1\right\rVert}

\newcommand{\ttnorm}[1]{\|#1\|}
\newcommand{\set}[1]{\left\{#1\right\}}							

\newcommand{\ttset}[1]{\{#1\}}

\newcommand{\cD}{\ensuremath{\mathcal D}} 
\newcommand{\cE}{\ensuremath{\mathcal E}}

\newcommand{\cL}{\ensuremath{\mathcal L}}

\newcommand{\cP}{\ensuremath{\mathcal P}}

\newcommand{\cS}{\ensuremath{\mathcal S}}

\author{Seonwoo Kim}
\address{Yonsei University}
\date{}
\email{seonwookim@yonsei.ac.kr}
\author{Matteo Quattropani}
\address{University of Roma Tre}
\date{}
\email{matteo.quattropani@uniroma3.it}
\author{Federico Sau}
\address{University of Milan}
\date{}
\email{federico.sau@unimi.it}

\begin{document}
	\title[Spectral gap of  stochastic exchange  models]{Spectral gap of the KMP
	 and other 
		 stochastic \\
		 exchange models on arbitrary graphs}
	\maketitle
	\begin{abstract}
We present a simple strategy to derive universal bounds on the spectral gap of reversible stochastic exchange models on arbitrary graphs. The Kipnis-Marchioro-Presutti (KMP) model, the harmonic process (HP), and the immediate exchange model (IEM) are all examples that fall into this class. Our upper and lower bounds depend only on two features: worst-case linear statistics and a kinetic factor, which is, in essence, graph-independent. For the three aforementioned examples, these bounds are sharp, and even saturate to an identity for HP and IEM in some log-concave regimes. The proof---which yields bounds for eigenvalues even in the non-reversible context---crucially exploits the rigidity of the eigenstructure of these models and  quantitative contraction rates of the corresponding hidden parameter models recently introduced in \cite{de_masi_ferrari_gabrielli_hidden_2023,giardina_redig_tol_intertwining_2024}.	\end{abstract}
	\setcounter{tocdepth}{1}
	\thispagestyle{empty}
	
\section{Introduction. From  KMP to more general exchange models}\label{sec:intro}
\subsection{The KMP model}
Consider a finite weighted	undirected graph $G=(V,(c_{xy})_{x,y\in V})$, some  positive parameters $\alpha=(\alpha_x)_{x\in V}$, and the following stochastic exchange dynamics involving non-negative
 masses $(\eta_x)_{x\in V}$ on $V$.
First, select  a pair of vertices $x,y\in V$ with rate $c_{xy}=c_{yx}\ge 0$. Then,  sample a random variable $U\sim {\rm Beta}(\alpha_x,\alpha_y)$. Finally, update  $(\eta_x,\eta_y)$ to 
 \begin{equation}\label{eq:KMP-update}
 	\tonde{U(\eta_x+\eta_y),(1-U)(\eta_x+\eta_y)}\comma
 \end{equation}
while leaving the other masses unchanged. This is known as the \emph{Kipnis-Marchioro-Presutti {\rm (KMP)} model} on the graph $G$ with parameters $\alpha$.

The original KMP model was introduced in \cite{kipnis_heat_1982} to describe heat conduction along a one-dimensional segment in contact with thermal reservoirs at its endpoints. In that setting, energy is exchanged only between neighboring vertices and with the reservoirs, and the model was defined with $U \sim \mathrm{U}(0,1)$ (i.e., $\alpha_x \equiv 1$). This formulation has since been generalized; see, e.g., \cite{giardina_duality_2009}.
On the same segment but without reservoirs, the model has been studied more recently under the name \emph{adjacent walk on the simplex} \cite{caputo_mixing_2019,labbe_petit_hydrodynamic_2025}. More broadly, this redistribution mechanism is also known as the \emph{beta-splitting process}, emphasizing the role of the $U$-distribution in determining how mass is split.
In the mean-field setting---where all rates $c_{xy}$ are equal and $\alpha_x \equiv 1/2$---the dynamics coincides with the evolution of particle energies in the \emph{Kac’s walk} \cite{kac_foundations_1956}, one of the most successful many-body mean-field models of stochastically colliding particles in kinetic theory.

In any of the above settings, the KMP model presents two  key qualitative features.
First, the dynamics described in \eqref{eq:KMP-update} is conservative, with the transition rates not depending on the total mass. Henceforth, without loss of generality, we shall always consider only unitary masses, i.e.,
\begin{equation}\label{eq:Omega}
	\eta\in \Omega\eqdef \set{\eta=(\eta_x)_{x\in V}: \eta_x\ge 0\ \text{for all}\ x\in V\,,\, |\eta|=1}\comma\quad \text{with}\ |\eta|=\sum_{x\in V}\eta_x\fstop
\end{equation} 
Second, the KMP model on $\Omega$ admits---as soon as the underlying graph is connected---a unique, explicit, reversible measure: the Dirichlet distribution with parameters $\alpha=(\alpha_x)_{x\in V}$, written as ${\rm Dir}(\alpha)$. Letting $\mu\in \cP(\Omega)$ denote this measure, this reads as
\begin{equation}\label{eq:mu-dirichlet}
	\mu(\dd \eta) = \frac1{{\rm B}(\alpha)}\bigg(\prod_{x\in V}\eta_x^{\alpha_x-1} \bigg)\dd \eta\comma\quad \text{with}\ {\rm B}(\alpha)\eqdef \frac1{\Gamma(|\alpha|)}\prod_{x\in V}\Gamma(\alpha_x)\comma
\end{equation}
where $\dd \eta$ stands for the uniform measure on $\Omega$. 

\subsubsection{Convergence to equilibrium and spectral gap}
Our primary focus is to quantify relaxation to equilibrium for this model on arbitrary graphs.	 
Because of the aforementioned reversibility, we shall focus on the process' spectral gap, governing the $L^2(\mu)$-exponential convergence rate. For more background on convergence to equilibrium, spectral gaps, and other functional inequalities for Markov processes, we refer to the monographs \cite{bakry_gentil_ledoux_analysis_2014,levin2017markov}.

The spectral gap of the KMP model has been completely characterized in the specific settings described above. Indeed, around two decades ago, in the mean-field setup,  spectral gap estimates---and even the exact value---have been obtained in a series of influential works \cite{janvresse_spectral_2001,carlen_carvalho_loss_determination_2003,caputo_kac2008}, thereby solving Kac's long-standing conjecture  \cite{kac_foundations_1956}. More recently, \cite{caputo_mixing_2019} determined the precise value of the spectral gap  on the  segment, whenever $\alpha_x\equiv{\rm const.}\ge 1$; for general $\alpha$,  bounds were previously derived in \cite{grigo_mixing_2012}. For mixing-time results, see, e.g., \cite{randall_winkler_mixing_2005,caputo_mixing_2019,labbe_petit_hydrodynamic_2025} on the segment, \cite{smith_gibbs_2014,caputo_quattropani_sau_universal_2025} for the mean-field setup, and references therein.

All the above spectral gap results crucially exploit inherent symmetries of the underlying geometry. When moving away from such settings, establishing the precise value of the spectral gap becomes unfeasible; however, one can still resort to spectral gap comparisons, that reduce an infinite- to a finite-dimensional optimization problem. Analogous spectral-gap reductions have been observed within the context of interacting particle systems; see Section \ref{sec:aldous-spectral-gap-identity} below for more details.

\subsubsection{${\rm KMP}$'s main result}
By reversibility, the KMP infinitesimal generator, say $\cL$, is self-adjoint in $L^2(\mu)$. In turn, the spectral gap of $\cL$ admits the following well-known variational formulation:
\begin{equation}\label{eq:gap-variational}
	{\rm gap}(G,\alpha) = \inf\set{\cD(f): f \in L^2(\mu)\ \text{with}\ \|f-\mu f\|_{L^2(\mu)}=1}\comma
\end{equation}
where $\cD(f)=\scalar{f}{-\cL f}_{L^2(\mu)}$ denotes the Dirichlet form associated to $\cL$, and $\mu f$ is a shorthand notation for $\int f\, \dd \mu$. However, thanks to the specifics of the model, the variational problem in \eqref{eq:gap-variational} can be considerably simplified. Indeed, as we shall explain below (see Proposition \ref{pr:invariance-polynomials}), $\cL$ leaves the (finite-dimensional) space $\mathscr P_k$ of polynomials of degree $k$ in the variables $(\eta_x)_{x\in V}$ invariant.  This, combined with the aforementioned reversibility with respect to $\mu$ in \eqref{eq:mu-dirichlet}, ensures that ${\rm gap}(G,\alpha)$ can be recovered by optimizing over each $\mathscr P_k\subset L^2(\mu)$ separately:
\begin{equation}\label{eq:gap-gap-k}
{\rm gap}(G,\alpha)= \inf_{k\ge 1} {\rm gap}_k(G,\alpha)\comma
\end{equation}
where, for all integers $k\ge 1$,
\begin{equation}\label{eq:gap-gap-k-def}
		{\rm gap}_k(G,\alpha) = \inf\set{\cD(f): f \in \mathscr P_k\ \text{with}\ \|f-\mu f\|_{L^2(\mu)}=1}\fstop
\end{equation}

Our main finding here is that ${\rm gap}(G,\alpha)$ and ${\rm gap}_1(G,\alpha)$ are comparable on \textit{any} underlying geometry, up to a universal factor, essentially depending only on $\alpha$.
\begin{theorem}[KMP]\label{th:gap-beta}
	For all graphs $G=(V,(c_{xy})_{x,y\in V})$ and  weights $\alpha=(\alpha_x)_{x\in V}$, 
	\begin{equation}\label{eq:gap-beta}
		\gamma_{\rm KMP}(G,\alpha)\, {\rm gap}_1(G,\alpha)\le {\rm gap}(G,\alpha)\le {\rm gap}_1(G,\alpha)\comma
	\end{equation}
	where
	\begin{equation}
		\gamma_{\rm KMP}(G,\alpha)\eqdef \frac{\alpha_{2,\rm min}}{1+\alpha_{2,\rm min}}\tonde{1+\frac1{|\alpha|}}\comma
		\end{equation}
		with
		\begin{equation}\label{eq:alpha-2-min}
		 \alpha_{2,\rm min}=\alpha_{2,\rm min}(G,\alpha)\eqdef \min_{x,y\,:\, c_{xy}>0} \set{\alpha_x+\alpha_y}\fstop
	\end{equation}
\end{theorem}
Given the relation in \eqref{eq:gap-gap-k} (which we will later prove in greater generality), the upper bound in \eqref{eq:gap-beta} is clearly not surprising; it is the lower bound that actually constitutes the main result of the theorem. However, taken together, they offer a clearer picture of the sharpness of our bounds. Indeed, one always has 
$\gamma_{\rm KMP}(G,\alpha) \in (0,1)$, as well as
	\begin{equation}\label{eq:gamma-beta-bounds}
	\frac{2\alpha_{\rm min}}{1+2\alpha_{\rm min}
	}\tonde{1+\frac1{|\alpha|}}	\le 	\gamma_{\rm KMP}(G,\alpha)\le \frac{2\alpha_{\rm max}}{1+2\alpha_{\rm max}}\tonde{1+\frac1{|\alpha|}}\comma
\end{equation}
where $\alpha_{\rm min}\eqdef\min_{x\in V}\alpha_x$ and $\alpha_{\rm max}\eqdef\max_{x\in V}\alpha_x$. Observe that both bounds in \eqref{eq:gamma-beta-bounds} do not depend on $G$. As a consequence of \eqref{eq:gamma-beta-bounds}, $\gamma_{\rm KMP}(G,\alpha)\approx 1$ (resp.\ $\frac{2}n$)  as $\alpha_{\rm min}\to \infty$ (resp.\ $\alpha_{\rm max}\to 0$ and $\frac{\alpha_{\rm min}}{\alpha_{\rm max}}\approx 1$). Finally, remark that $\gamma_{\rm KMP}(G,\alpha)$ can be much better than its lower bound in \eqref{eq:gamma-beta-bounds}, whenever only a few isolated vertices $x\in V$ register $\alpha_x\to 0$.	

Remark that working with ${\rm gap}_1(G,\alpha)$ instead of ${\rm gap}(G,\alpha)$ offers a significant advantage: as we will see, ${\rm gap}_1(G,\alpha)$ admits a natural interpretation as the spectral gap of a random walk on a weighted graph. For such an object, a range of efficient estimation techniques from the theory of Markov chains and graph theory are available, and in some cases, even explicit formulas can be derived. This is precisely the case when $\alpha_x \equiv \text{const.} = \hat{\alpha} > 0$ and $G$ is either the complete graph or the segment. Comparing these explicit formulas with the KMP's spectral gap identities  in \cite{carlen_carvalho_loss_determination_2003,caputo_mixing_2019} allow us to assess the sharpness of Theorem \ref{th:gap-beta} (Section \ref{sec:KMP-sharpness}): in the complete-graph case, the lower bound in \eqref{eq:gap-beta} becomes an identity; on the segment, the upper bound is saturated whenever $\hat \alpha\ge 1$. 

\subsection{Proof strategy and hidden parameter model}\label{sec:intro-proof}
As already mentioned, prior to our work, spectral gap lower bounds for the KMP model have been established on the complete graph and on the segment.
In particular, the first setting---corresponding to the mean-field scenario, much explored in kinetic theory---is now fully covered \cite{janvresse_spectral_2001,carlen_carvalho_loss_determination_2003,caputo_kac2008}. In a nutshell, the common ingredient of all these proofs is a recursive argument on $n$, the size of the graph, at the level of Dirichlet forms. This strategy, originated from Lu-Yau's martingale method \cite{lu_yau_spectral_1993} and which proved to be extremely efficient and flexible to generalizations in the mean-field setup \cite{carlen_spectral_2014,sasada_spectral_2015,carlen_posta_toth_spectral_2025}, does not  admit a sharp counterpart on \textit{arbitrary} underlying geometries, especially in absence of symmetries.

Instead, we follow an alternative, equally established, route due to Chen and Wang \cite{chen1997estimation} to bound from below eigenvalues of Markov generators, thus, spectral gaps of reversible processes.  In essence, their proof goes as follows. Let $\lambda>0$ be an eigenvalue of $-\cL$ associated to the eigenfunction  $f$ satisfying $\mu f=0$, and $(\cS_t)_{t\ge 0}$ its corresponding  semigroup. Then,   one has,	for any $\eta\in \Omega$ and $t> 0$, 
\begin{equation}
	\ttabs{e^{-\lambda t}f(\eta)}=\tabs{\cS_t f(\eta)}
	=\ttabs{\cS_t f(\eta)-\mu \cS_t f}\le {\int_\Omega \mu(\dd \eta') \ttabs{\cS_t f(\eta)-\cS_t f(\eta')}}\comma
\end{equation}
and the task is to bound the right-hand side  by $Ce^{-\varrho t}$, for all $t>0$ large enough and some constants $C,\varrho>0$ not depending on $t>0$. Provided that $f(\eta)\neq 0$, taking the logarithm on both sides and sending $t\to \infty$ would yield the lower bound $\lambda \ge \varrho$. 

In fact, the crucial estimate
\begin{equation}\label{eq:chen-wang}
	\int_\Omega\mu(\dd \eta')\ttabs{\cS_t f(\eta)-\cS_tf(\eta')}\le Ce^{-\varrho t}
\end{equation} is derived in two main steps, see, e.g., \cite[Theorem 2.42]{chen_eigenvalues_2005}: 
\begin{enumerate}[(i)]
	\item Find a bounded nonnegative function $\mathfrak d$ on $\Omega\times \Omega$ controlling $\ttabs{f(\emparg)-f(\emparg)}$, that is, 
	\begin{equation}
		\sup_{\eta,\eta'\in\Omega}\frac{\ttabs{f(\eta)-f(\eta')}}{\mathfrak d(\eta,\eta')}\le C_1\comma\quad \text{for some}\ C_1>0\fstop
		\end{equation}
		Clearly, this forces $\mathfrak d(\eta,\eta)=0$ for all $\eta\in \Omega$.
		\item Find a Markovian coupling $\tilde \cS_t$ of two copies of $\cS_t$ under which $\mathfrak d$ decays exponentially fast, i.e., there exists $\varrho>0$ satisfying, for all $t>0$ and $\eta,\eta'\in \Omega$,
		\begin{equation}\label{eq:exp-contraction-general}
			\tilde \cS_t \mathfrak d(\eta,\eta')\le e^{-\varrho t}\,\mathfrak d(\eta,\eta')\fstop
		\end{equation}
		In particular, when $\mathfrak d$ is the $p$-th power of a distance on $\Omega$, \eqref{eq:exp-contraction-general} gives a $p$-Wasserstein (with respect to $\mathfrak d^{1/p}$) bound.
\end{enumerate}
These two estimates yield \eqref{eq:chen-wang} with $C=C_1 \sup_{\eta,\eta'\in\Omega}\mathfrak d(\eta,\eta')>0$. The challenge is, thus, to devise such a (i) control function  and (ii) coupling.

\subsubsection{Chen-Wang's approach for ${\rm KMP}$}
Within the context of the KMP model, this program has been successfully carried out, for $\alpha\equiv {\rm const.}$, on the complete graph \cite{hauray_uniform_2016} and on the segment \cite{grigo_mixing_2012}. 
Both works employ  the so-called \textit{grand-coupling}: let  the two configurations evolve as in \eqref{eq:KMP-update} by simultaneously  updating the same edge and sharing the same $U$. 
As for  $\mathfrak d$, \cite{hauray_uniform_2016} considers the squared Euclidean distance on $\Omega$: 
\begin{equation}\label{eq:euclidean-distance}
	\mathfrak d(\eta,\eta')=d(\eta,\eta')^2\eqdef\sum_{x\in V}\tonde{\eta_x-\eta_x'}^2\fstop
	\end{equation}
It is instructive to check what  \cite{hauray_uniform_2016} yields on a general graph $G=(V,(c_{xy})_{x,y\in V})$: letting $\tilde \cL$ denote the generator of the semigroup $(\tilde \cS_t)_{t\ge 0}$, by following the steps in the proof of \cite[Theorem 1]{hauray_uniform_2016}, one gets (recall that here we assume $\alpha\equiv \hat \alpha>0$)
\begin{align}\label{eq:hauray}
	\begin{aligned}
	\tilde \cL\, d(\eta,\eta')^2 &=-\frac{\hat\alpha}{4\hat\alpha+2}\sum_{\substack{x,y\in V\\
	x\neq y}}c_{xy}\set{\tonde{\eta_x-\eta_x'}-\tonde{\eta_y-\eta_y'}}^2\\
&\quad + \frac{1}{2\hat\alpha+1}\sum_{\substack{x,y\in V\\
x\neq y}}c_{xy}\tonde{\eta_x-\eta_x'}\tonde{\eta_y-\eta_y'}\fstop
\end{aligned}
\end{align}
On the complete graph, that is, when $c_{xy}\equiv \frac1n>0$, both summations on the right-hand side can be reduced to a negative constant---depending on $|V|=n$ and $\hat \alpha$---times $d(\eta,\eta')^2$. Here, one crucially uses that $\sum_{y\neq x}c_{xy}\tonde{\eta_y-\eta_y'}=-\frac1n\tonde{\eta_x-\eta_x'}$.  Recalling that $\tilde \cL \tilde \cS_t = \frac{\dd}{\dd t}\tilde \cS_t$, integrating this infinitesimal identity yields the desired exponential contraction as in \eqref{eq:exp-contraction-general}. Interestingly, the spectral gap lower bound obtained in this way matches the exact value \cite{carlen_carvalho_loss_determination_2003,caputo_kac2008}.

When $G$ is not the complete graph, this simplification breaks down. In particular, while the first expression  is nonpositive and eventually comparable to $d(\eta,\eta')^2$, the second line in \eqref{eq:hauray} does not have, in general, a prescribed sign. This represents the main obstacle in making this strategy work on arbitrary graphs, as different couplings and control functions need to be constructed. 

On the segment, for instance,  \cite{grigo_mixing_2012}  overcomes this issue by  replacing the squared  distance of $\eta, \eta'\in \Omega$ in \eqref{eq:euclidean-distance} with that of their height functions, namely, 
\begin{equation}
	\mathfrak d(\eta,\eta')=\sum_{x\in V}\bigg(\sum_{y\le x}\eta_y-\sum_{y\le x}\eta_y'\bigg)^2\fstop
\end{equation} 
This choice is clearly effective only on one-dimensional graphs.
For more details, we refer the interested reader to \cite[Proposition 2.6]{grigo_mixing_2012}, as well as to \cite{caputo_mixing_2019,labbe_petit_hydrodynamic_2025}, which determine the spectral gap for some choices of $\alpha$, with $\alpha_{\rm min}\ge 1$.

\subsubsection{${\rm KMP}$'s hidden parameter model} Our way to tackle this difficulty on arbitrary graphs is to apply the same strategy, not to the KMP process itself, but to a recently discovered dual process \cite{de_masi_ferrari_gabrielli_hidden_2023} (see also \cite{giardina_redig_tol_intertwining_2024}): the KMP hidden parameter model.
Before discussing the details, we begin with a few preliminary remarks.

Recall that $\cL$ leaves $\mathscr P_k$,  the space of polynomials of degree $k$ in the $\eta$-variables, invariant. More specifically, let $f\in \mathscr P_k$, and represent it as
\begin{equation}\label{eq:f-polynomial}
	f(\eta)=\sum_{x_1,\ldots,x_k\in V} \psi(x_1,\ldots, x_k)\,\eta_{x_1}\cdots \eta_{x_k}\comma\qquad \eta \in \Omega\comma
\end{equation} 
for some symmetric function $\psi:V^k\to \R$. Since $\cL$ is linear and $\eta$ vary in $\Omega$, the action of $\cL$ on $f=f(\eta)$ translates into one on the coefficients $\psi=\psi(x_1,\ldots,x_k)$, that is, there exists a linear operator $L$ on $\R^{V^k}$ satisfying, for all $\eta\in \Omega$,
\begin{equation}
	\cL f(\eta)=\sum_{x_1,\ldots,x_k\in V}\psi(x_1,\ldots, x_k)\, \cL\,{\eta_{x_1}\cdots \eta_{x_k}}= \sum_{x_1,\ldots, x_k\in V} L\psi(x_1,\ldots, x_k)\, \eta_{x_1}\cdots \eta_{x_k}\fstop
\end{equation}
As we shall see in Section \ref{sec:particle-system}, and as is well-known, $L$ has a probabilistic interpretation in terms of a particle dynamics, and the identity above is an instance of an intertwining relation. What matters for us now is that, since any eigenfunction of $-\cL$ must  be a polynomial, this implies that, if $f$ in \eqref{eq:f-polynomial} solves $\cL f=-\lambda f$, then so do its coefficients, i.e., 
\begin{equation}\label{eq:psi-eigenfunction}
	L\psi=-\lambda \psi\fstop
\end{equation}
Exploiting this equivalent reformulation of the eigenvalue equation for KMP seems as hard as the one we started from. Instead, we combine the eigenfunction $\psi$ in \eqref{eq:psi-eigenfunction} with different variables, as follows
\begin{equation}
	g(\theta)\eqdef \sum_{x_1,\ldots,x_k\in V}\alpha_{x_1\cdots x_k}\, \psi(x_1,\ldots, x_k)\, \theta_{x_1}\cdots \theta_{x_k}\comma\qquad \theta \in [0,1]^V\comma
\end{equation}
where $\alpha_{x_1\cdots x_k}\eqdef \alpha_{x_1}\tonde{\alpha_{x_2}+\car_{x_2}(x_1)}\cdots \tttonde{\alpha_{x_k}+\sum_{j<k}\car_{x_k}(x_j)}$. This seemingly marginal change turns out to be surprisingly effective.

The action of the linear operator $L$ on $\psi$ translates into a new one on the $\theta$-variables, which we call $\mathscr L$: for all $\theta\in [0,1]^V$, 
\begin{align}
\sum_{x_1,\ldots, x_k} \alpha_{x_1\cdots x_k}\, L\psi(x_1,\ldots, x_k)\, \theta_{x_1}\cdots \theta_{x_k} 
&= \mathscr Lg(\theta)\fstop
\end{align}
Hence, if $\psi$ solves \eqref{eq:psi-eigenfunction}, then $\mathscr L g=-\lambda g$. Altogether, we obtained
\begin{equation}
	\cL f=-\lambda f\qquad \Rightarrow\qquad L \psi=-\lambda \psi\qquad \Rightarrow\qquad \mathscr L g=-\lambda g\fstop
\end{equation}

The operator $\mathscr L$ just introduced turns out to be the infinitesimal generator of the hidden parameter model  of the KMP. Therefore, letting  $(\mathscr S_t)_{t\ge 0}$ be its Markov semigroup, by the same reasoning as above, we obtain a lower bound for $\lambda$ if we prove
\begin{equation}\label{eq:ub-HPM}
\ttabs{e^{-\lambda t}g(\theta)}=	\ttabs{\mathscr S_t g(\theta)} \le Ce^{-\varrho t}\comma
\end{equation}
 for all $t>0$ large enough, and for some $C, \varrho>0$ and $\theta\in [0,1]^V$ satisfying $g(\theta)\neq0$.

For the proof of the upper bound in \eqref{eq:ub-HPM}, we exploit the probabilistic interpretation of $\mathscr L$ and $\mathscr S_t$.
On a given graph,  KMP and its hidden parameter model $\theta$  evolve according to the same stochastic mechanisms: edges are selected at the same rate, and an analogous beta-distributed random variable 
$U$ is sampled. However, the resulting $\theta$-configuration is not obtained from the update rule in \eqref{eq:KMP-update}; instead, it transitions from $(\theta_x,\theta_y)$ to
\begin{equation}\label{eq:HPM-update}
	(U\theta_x+(1-U)\theta_y,U\theta_x+(1-U)\theta_y)\fstop
\end{equation}
Remark that \eqref{eq:HPM-update} describes a dynamics that:
\begin{itemize}
	\item does not, in general, conserve the total mass $|\theta|=\sum_{x\in V}\theta_x$;
	\item satisfies the maximum principle: if $
	\max_{x\in V}\theta_x\le 1$ holds at time $t=0$, then this holds true almost surely at any later time;
	\item leaves constant configurations invariant; moreover, if $G$ is connected, these are the only invariant ones. 
\end{itemize}
This last feature plays a crucial role in our analysis. While having equilibrium states supported only on constant configurations may appear to introduce a degeneracy, it is in fact  beneficial: convergence to equilibrium for the hidden parameter model  requires no coupling, but only a control over spatial fluctuations. 
The  decay of these fluctuations is quantified by the following $L^2$-estimate (here, $\pi=(\pi_x)_{x\in V}$ with $\pi_x=\frac{\alpha_x}{|\alpha|}$): for all $t\ge 0$,
\begin{equation}\label{eq:ub-HPM2}
	(\mathscr S_t{\rm Var}_\pi)(\theta)\le e^{-\varrho t} \, {\rm Var}_\pi(\theta)\comma
\end{equation}
with $\theta\mapsto {\rm Var}_\pi(\theta)=\|\theta-\pi\theta\|_{L^2(\pi)}^2$ being the suitable control function for $k$-degree eigenfunctions, $k\ge 2$	(Section \ref{sec:proof-control}). Ultimately, establishing estimates as in \eqref{eq:ub-HPM2} boils down to a straightforward second moment computation (Section \ref{sec:proof-L2}). Remarkably, these estimates are broadly applicable: they are graph- and coupling-independent, and yield tight bounds across a wide class of processes.

\subsection{Generalizations: HP \&  IEM, and beyond reversibility}

The strategy  outlined for bounding KMP's spectral gap turns out to be much more general, yet, surprisingly sharp. Indeed, it covers all	conservative stochastic exchange models that are reversible,  and that come with a \textquotedblleft degenerate\textquotedblright\ hidden parameter model. In certain instances, the spectral gap estimates derived through this approach actually become exact identities, on any graph.
	
Stochastic exchange models which are conservative and admit a hidden parameter model take a rather general form, see \eqref{eq:gen-general}: transitions correspond to  stochastic mappings, with the corresponding rates being independent of the process' state. 	
In contrast, reversibility imposes a specific form on the transition rates. 

\subsubsection{Harmonic process \& immediate exchange model} 

We mainly consider two notable examples in this class: the harmonic process (HP) \cite{frassek2021exact,franceschini_frassek_giardina_integrable_2023} and the immediate exchange model (IEM) \cite{katriel_immediate_2015,ginkel_redig_sau_duality_2016}. A rough description of their dynamics on $(G,\alpha)$ goes   as follows (see Sections \ref{sec:general} and \ref{sec:examples} for a complete description): 
 HP's transitions go from $(\eta_x,\eta_y)$ to
\begin{equation}
	(U\eta_x,(1-U)\eta_x+\eta_y)\comma
\end{equation}
where the random $U$ takes the value $u\in (0,1)$ with an \textquotedblleft infinitesimal rate\textquotedblright, namely, $c_{xy}\tonde{1-u}^{-1}u^{\alpha_x-1}\,\dd u$;	 
IEM's updates result in replacing, at rate $c_{xy}\ge 0$, $(\eta_x,\eta_y)$ with
\begin{equation}\label{eq:update-IEM}
	(U\eta_x+(1-V)\eta_y,(1-U)\eta_x+V\eta_y)\comma
\end{equation}
where $U, V\in (0,1)$ are independent and distributed as
	\begin{equation}U\sim {\rm Beta}(\alpha_x-\kappa,\kappa)\comma\qquad V\sim{\rm Beta}(\alpha_y-\kappa,\kappa)\fstop\end{equation}
	Here,  $\kappa \in (0,\alpha_{\rm min})$ is an additional fixed parameter of the model.

 HP and IEM fall into our general class of stochastic exchange models. Furthermore, because both models admit $\mu \sim {\rm Dir}(\alpha)$ as a reversible measure, the spectral gap decomposition in \eqref{eq:gap-variational}--\eqref{eq:gap-gap-k-def} holds analogously for both models. 
In the following theorems, we specialize the main results from Section \ref{sec:results} to the context of HP and IEM. Observe that, although we simply write ${\rm gap}(G,\alpha)$ and ${\rm gap}_1(G,\alpha)$ as in Theorem \ref{th:gap-beta}, these quantities depend, in general, on the model considered. All throughout, $a\wedge b\eqdef \min\ttset{a,b} $, $a,b\in \R$.

\begin{theorem}[HP]\label{th:gap-HP} For all graphs $G=(V,(c_{xy})_{x,y\in V})$ and weights $\alpha=(\alpha_x)_{x\in V}$, 
	\begin{equation}\label{eq:HP-gap-result}
		\tonde{1\wedge \gamma_{\rm HP}(G,\alpha)}{\rm gap}_1(G,\alpha)\le {\rm gap}(G,\alpha)\le {\rm gap}_1(G,\alpha)\comma
	\end{equation}
	where
	\begin{equation}
		\gamma_{\rm HP}(G,\alpha)\eqdef \min_{x,y\,:\,c_{xy}>0} \set{\frac{\alpha_x}{\alpha_x+1}+\frac{\alpha_y}{\alpha_y+1}}\tonde{1+\frac1{|\alpha|}}\fstop
	\end{equation}
\end{theorem}
\begin{theorem}[IEM]\label{th:gap-IEM} For all  $G=(V,(c_{xy})_{x,y\in V})$,  $\alpha=(\alpha_x)_{x\in V}$, and $\kappa\in (0,\alpha_{\rm min})$, 
	\begin{equation}\label{eq:IEM-gap-result}
		\tonde{1\wedge \gamma_{\rm IEM}(G,\alpha,\kappa)}{\rm gap}_1(G,\alpha,\kappa)\le {\rm gap}(G,\alpha,\kappa)\le {\rm gap}_1(G,\alpha,\kappa)\comma
	\end{equation}
	where
	\begin{equation}
		\gamma_{\rm IEM}(G,\alpha,\kappa)\eqdef \min_{x,y\,:\,c_{xy}>0} \set{\frac{\alpha_x-\kappa}{\alpha_x+1}+\frac{\alpha_y-\kappa}{\alpha_y+1}}\tonde{1+\frac1{|\alpha|}}\fstop
	\end{equation}
\end{theorem}

Strikingly, if $\alpha_{\rm min}\ge 1$ (resp.\ $\alpha_{\rm min}\ge 1+2\kappa$), then the bounds in \eqref{eq:HP-gap-result} (resp.\ \eqref{eq:IEM-gap-result}) are actually identities, that is, ${\rm gap}(G,\alpha)={\rm gap}_1(G,\alpha)$.
 Whenever this identity does not hold, we prove that the lower bounds are attained on specific graphs; see Sections \ref{sec:HP-sharpness} and \ref{sec:IEM-sharpness}. Observe that the regime $\alpha_{\rm min}\ge 1$ corresponds to log-concavity of $\mu\sim {\rm Dir}(\alpha)$. 
 
Finally, let us note that the homogeneous HP model on the segment is known to be integrable \cite{frassek2021exact}. Consequently, all eigenvalues of the HP model in this specific setting are, at least in principle, explicitly computable. However,  no results on the spectral gap of either of these two models were available, in any geometry, prior to the inequalities and identities stated in Theorems \ref{th:gap-HP} and \ref{th:gap-IEM}.

\subsubsection{Beyond non-degeneracy and reversibility}

All three models discussed earlier share the key feature of admitting a non-degenerate reversible  measure. These two properties are instrumental in applying the spectral theorem to self-adjoint generators, which in turn guarantees that the polynomial eigenvalues we compute actually saturate the corresponding $L^2$-spectrum. However, one of the main strengths of our approach is that neither non-degeneracy nor reversibility is essential for obtaining sharp eigenvalue estimates. Indeed, we identify a broad family of stochastic exchange models—with update rules as in \eqref{eq:update-IEM}, see \eqref{eq:gen-general}—for which our method remains effective even in the absence of these assumptions, thus significantly extending the range of models to which our spectral analysis applies.

To begin with, our strategy remains effective in degenerate settings, where the invariant measure is singular, that is, $\mu=\delta_\omega$, for some $\omega\in \Omega$. This occurs, for example, in certain averaging-type models in which one imposes the requirement \begin{equation}\label{eq:averaging}
	V=1-\frac{\omega_x}{\omega_y}\tonde{1-U}
\end{equation} at each update in \eqref{eq:update-IEM}. In such cases, the system  converges to the  configuration $\omega\in \Omega$; when $\omega\equiv \frac1n$, the above requirement becomes $U\equiv V$.

In these settings, no meaningful $L^2$-space compatible with $\cL$ can be defined; nevertheless, our method remains applicable and yields explicit lower bounds on eigenvalues associated with polynomial observables. These bounds are not merely formal. As shown in Section~\ref{sec:particle-system}, for each integer $k\ge 1$, a stochastic exchange model naturally gives rise to a $k$-particle  interacting system. In the specific case described in \eqref{eq:averaging}, it is straightforward to check that these particle systems always admit a fully supported reversible measure. Consequently,  our eigenvalue estimates yield lower bounds on the spectral gap of the corresponding particle systems, uniformly in the number of particles, $k\ge 1$. We note that the averaging model with $U \equiv \frac{\omega_x}{\omega_x + \omega_y}$ in \eqref{eq:averaging} was considered in \cite{quattropani2021mixing}.

Even more significantly, our approach yields meaningful lower bounds on the real part of eigenvalues for a broad class of non-reversible stochastic exchange models and their associated particle systems. This class includes models with asymmetric updates, as well as models arising in econophysics, where the conserved quantity is interpreted as wealth and the dynamics incorporates a nonzero saving propensity. As shown in \cite{cirillo_redig_ruszel_duality_2014}, such models are intrinsically non-reversible.

\subsubsection{Open systems}
Hidden parameter models for the KMP and the HP recently sparked  attention  within the context of open systems, in contact with external reservoirs, providing a useful alternative description of the corresponding non-equilibrium steady states \cite{de_masi_ferrari_gabrielli_hidden_2023,carinci_solvable_2024,carinci_large_2025,giardina_redig_tol_intertwining_2024}.
Each of our general stochastic exchange models admits a non-conservative counterpart, with suitably chosen boundary-driven dynamics. The polynomial-eigenvalue analysis presented in \cite[Section 6]{kim_sau_one_2024} for a related open diffusion process can be similarly applied to our models, yielding analogous results---specifically, that the lowest polynomial eigenvalue is attained by linear functions. For the sake of brevity, we will not discuss this setting in any detail.

\subsection{Organization of the paper} The rest of the paper is organized as follows. In Section \ref{sec:general}, we introduce a general class of stochastic exchange models,  discuss some of their basic properties, and present our main results on eigenvalue and spectral gap estimates. Section \ref{sec:preliminary} presents definitions and main properties of three Markov processes related to our stochastic exchange models. Section \ref{sec:proof} contains the core of the proof of the main results from Section \ref{sec:general}. Section \ref{sec:examples} presents a detailed discussion of the three models introduced so far and illustrates our main results by applying them to prove Theorems \ref{th:gap-beta}--\ref{th:gap-IEM}.

\section{General setting and main results}\label{sec:general}
\subsection{Stochastic exchange models}\label{sec:SEM}
Consider $|V|=n\in \N$ vertices, $n\ge 2$.
Our configuration space is the probability simplex $\Omega$ defined in \eqref{eq:Omega}. The general class of models that we consider is encoded in the  (formal) generator	 (we always assume $\beta_{xx}\equiv 0$) 
\begin{equation}\label{eq:gen-general}
	\cL f(\eta)= \sum_{x,y\in V}\int_{[0,1]^2} \beta_{xy}(\dd u,\dd v) \set{f(\eta R_{xy}^{uv})-f(\eta)}\comma\qquad \eta\in \Omega\comma
\end{equation}
which acts on functions $f$ on $\Omega$ (see Section \ref{sec:assumptions} for precise definitions and assumptions).

 Here, $\beta_{xy}$ is a Borel measure on $[0,1]^2$, while   $R_{xy}^{uv}$ is the $n\times n$ stochastic matrix with $1$'s on each diagonal $z$-th entry, $z\neq x,y$, and having the block 
\begin{equation}\label{eq:Ruv}
	R^{uv}=\begin{pmatrix}
		u&1-u\\
		1-v&v
	\end{pmatrix}
\end{equation}
in the two-by-two sub-matrix corresponding to the coordinates $x,y \in V$. More explicitly, 
\begin{equation}\label{eq:Rxyuv}
	R_{xy}^{uv}= \varPi_{xy} \footnotesize\begin{pmatrix}
		u &1-u &0 &\cdots &\cdots&0\\
		1-v&v & 0 &\cdots&\cdots &0\\
		0 &0&1 &0 &\cdots&0\\
		\vdots& &&\ddots&&\vdots\\
		&&&&\ddots&0\\
		0&\cdots&\cdots&\cdots&\cdots&1
	\end{pmatrix} \normalsize\varPi_{xy} = \varPi_{xy}
	\renewcommand{\arraystretch}{1.5}
	\begin{pNiceArray}{c|c}
		R^{uv} & 0 \\
		\hline
		0 & \mathds 1_{n-2}
	\end{pNiceArray}
	\varPi_{xy}\comma
\end{equation}
where $\varPi_{xy}=\varPi_{xy}^T$ is the $n\times n$ permutation matrix exchanging the first (resp.\  second) row with the $x$-th (resp.\ $y$-th) one, and $\mathds 1_{n-2}$ is the identity matrix on $\R^{V\setminus \{1,2\}}$.	Hence, in \eqref{eq:gen-general}, we shall interpret $\eta R_{xy}^{uv}$ as a row vector-matrix multiplication: recalling \eqref{eq:Ruv}--\eqref{eq:Rxyuv}, 
\begin{align}\label{eq:SEM-row-matrix}
	\begin{aligned}
	 ((\eta R_{xy}^{uv})_x,(\eta R_{xy}^{uv})_y) = (\eta_x,\eta_y) R^{uv}
	 &= (\eta_x,\eta_y)\begin{pmatrix}
		u&1-u\\
		1-v&v
	\end{pmatrix}\\
	&= (u\eta_x+(1-v)\eta_y,(1-u)\eta_x+v\eta_y)\comma
	\end{aligned}
\end{align}
whereas $	(\eta R_{xy}^{uv})_z=\eta_z$, $z\neq x, y$. 
\begin{remark}[Hidden parameter model]\label{rem:HPM}
	In contrast, \textit{hidden parameter models} are constructed from matrix-column vector multiplications. In formulas, the corresponding generator informally reads, for functions $g:[0,1]^V\to \R$, as
	\begin{equation}\label{eq:gen-HPM}
		\mathscr L g(\theta)= \sum_{x,y\in V}\int_{[0,1]^2}\beta_{xy}(\dd u,\dd v)\set{g(R_{xy}^{uv}\theta)-g(\theta)}\comma\qquad \theta \in [0,1]^V\fstop
	\end{equation}
	Notice that, on right-hand side of \eqref{eq:gen-HPM}, we have $R_{xy}^{uv}\theta$ (with $\theta$ interpreted as a column vector), whereas in \eqref{eq:gen-general} we wrote $\eta R_{xy}^{uv}$ (with $\eta$ interpreted as a row vector). By comparing with \eqref{eq:SEM-row-matrix}, an update for the hidden parameter model reads as
	\begin{align}\label{eq:HPM-update-matrix}
		\begin{pmatrix}
			(R_{xy}^{uv}\theta)_x\\
			(R_{xy}^{uv}\theta)_y
		\end{pmatrix} = R^{uv}\begin{pmatrix} \theta_x\\
		\theta_y		
		\end{pmatrix}=\begin{pmatrix}
		u &1-u\\
		1-v&v
		\end{pmatrix}\begin{pmatrix} \theta_x\\
		\theta_y		
		\end{pmatrix} = \begin{pmatrix}
		u\theta_x+(1-u)\theta_y\\
		(1-v)\theta_x+v\theta_y
		\end{pmatrix}\fstop
	\end{align}
	More details on hidden parameter models are discussed in Section \ref{sec:HPM}.
\end{remark}
	
\begin{remark}The models discussed in Section \ref{sec:intro} arise from specific choices of $\beta_{xy}$ in \eqref{eq:gen-general}:
	\begin{align}
		\beta_{xy}(\dd u,\dd v)&\propto c_{xy}\,\ttset{ u^{\alpha_x-1}\tonde{1-u}^{\alpha_y-1} \delta_{1-u}(\dd v) }\,\dd u\comma	\\
		\beta_{xy}(\dd u,\dd v)&\propto c_{xy}\,\ttset{	u^{\alpha_x-1}\tonde{1-u}^{-1}\, \delta_1(\dd v)}\,\dd u\comma
		\\
		\beta_{xy}(\dd u,\dd v)&\propto c_{xy}\,\ttset{u^{\alpha_x-\kappa-1}\tonde{1-u}^{\kappa-1} v^{\alpha_y-\kappa-1}\tonde{1-v}^{\kappa-1}}\, \dd u\dd v\comma
	\end{align}
	correspond to ${\rm KMP}$, ${\rm HP}$, and ${\rm IEM}$, respectively (see Section \ref{sec:examples} for further details). 
\end{remark}

\subsection{Precise assumptions and first properties}\label{sec:assumptions}
Let $\mathscr P_0\eqdef \set{f:\Omega\to \C: f\equiv {\rm const.}}$ and, for all  integers $k\ge 1$, 
\begin{equation}\label{eq:Pk}
	\mathscr P_k\eqdef \set{f:\Omega\to\C: f\ \text{polynomial with}\ {\rm deg}(f)\le k}\comma\qquad\mathscr P\eqdef \cup_{k\ge 0}\mathscr P_k \fstop
\end{equation}
Clearly, $\mathscr P_{k-1}\subset \mathscr P_k$. Furthermore, we define
\begin{equation}\label{eq:Pk-dagger}
	\mathscr P_{k,\dagger}\eqdef \mathscr P_k\setminus \mathscr P_{k-1}\fstop
\end{equation}

The following assumption on the Borel measures $(\beta_{xy})_{x,y\in V}$ on $[0,1]^2$  is a minimal one to guarantee that $\cL$ is well-defined on polynomials and that eventually generates a Feller process.
\begin{assumption}\label{ass:beta}
	For all  $x,y \in V$, we have $\beta_{xx}\equiv 0$ and
	\begin{equation}\label{eq:ass-beta-xy}
		\int_{[0,1]^2}\beta_{xy}(\dd u,\dd v)\set{\tonde{1-u}+\tonde{1-v}}<\infty\fstop
	\end{equation}
	Further, we  set $\cL1=0$, and therefore do	 \textit{not} require  $\beta_{xy}([0,1]^2)=\int_{[0,1]^2}\beta_{xy}(\dd u,\dd v)<\infty$.
\end{assumption}

\begin{proposition}[Well-definition]\label{pr:invariance-polynomials}
	Under Assumption \ref{ass:beta}, $\cL$ is well-defined on $\mathscr P$ and satisfies $\cL\mathscr P_k\subset \mathscr P_k$. Moreover,  $\cL\restr{\mathscr P}$ is the pre-generator of a Feller process on $\Omega$.
\end{proposition}
\begin{proof}
	For all $x,y, x_1,\ldots, x_k \in V$, $k\ge 1$, and $u,v\in [0,1]$, the expression 
	\begin{equation}
		\eta\in \Omega\longmapsto \ttset{	(\eta R_{xy}^{uv})_{x_1}\cdots (\eta R_{xy}^{uv})_{x_k}-\eta_{x_1}\cdots \eta_{x_k}}
	\end{equation}
	obtained by taking $\eta\mapsto f(\eta)=\eta_{x_1}\cdots \eta_{x_k}$ in the  curly brackets of \eqref{eq:gen-general}
	is clearly a  polynomial in the $\eta$-variables of degree at most $k$. Hence, this and linearity imply	 $\cL \mathscr P_k\subset \mathscr P_k$,  \textit{provided that $\cL f=\cL\,{\eta_{x_1}\cdots \eta_{x_k}}$ is well-defined}. 
	This last condition holds by Assumption \ref{ass:beta}.  Indeed, since $\ttabs{(\eta R_{xy}^{uv})_z-\eta_z}\le (1-u)\eta_x+(1-v)\eta_y$ and $0\le \eta_z,(\eta R_{xy}^{uv})_z\le 1$ for all $z\in V$, a simple telescopic argument yields
	\begin{equation}
		\ttabs{	(\eta R_{xy}^{uv})_{x_1}\cdots (\eta R_{xy}^{uv})_{x_k}-\eta_{x_1}\cdots \eta_{x_k}}
		\le k \set{(1-u)+(1-v)}\fstop
	\end{equation}
	By taking the integral with respect to $\beta_{xy}$ on both sides and summing  over $x,y \in V$, linearity and 	Assumption \ref{ass:beta} ensure that $\cL\restr{\mathscr P_k}$ is well-defined. 
	This proves the first claim.
	
	The second claim follows a standard line of reasoning; we refer to \cite[Chapter I.2]{liggett_interacting_2005-1} for the details.	
\end{proof}

In our analysis, the action of $\cL$ on linear functions $f\in \mathscr P_1$ plays a central role. A simple computation yields, for $f(\eta)=\sum_{x\in V}\psi(x)\, \eta_x$, $\psi \in \R^V$, 
	\begin{align}
	\cL f(\eta)&=\sum_{x\in V}\psi(x)\sum_{y\neq x} {\int_{[0,1]^2}\quadre{\beta_{xy}(\dd u,\dd v)+\beta_{yx}(\dd v,\dd u)}\set{(1-v)\eta_y-(1-u)\eta_x}
	}\\
	&= \sum_{x\in V}\set{\sum_{y\neq x}r_{xy}\set{\psi(y)-\psi(x)}}\eta_x\comma\label{eq:linear}
\end{align}
where we defined, for all distinct  $x,y \in V$,
	\begin{align}
		r_{xy}&\eqdef \int_{[0,1]^2} \beta_{xy}(\dd u,\dd v)\tonde{1-u} + \int_{[0,1]^2}\beta_{yx}(\dd v,\dd u)\tonde{1-u}\\
		&= \int_{[0,1]^2} \quadre{\beta_{xy}(\dd u,\dd v)+\beta_{yx}(\dd v,\dd u)}\tonde{1-u}\fstop
		\label{eq:rates-rxy}
\end{align}	
Letting $r_{xx}\eqdef -\sum_{y\neq x}r_{xy}$, $x\in V$, we readily interpret the expression between curly brackets in \eqref{eq:linear} as the action on $\psi\in \R^V$ of the generator of a continuous-time random walk   with rates $(r_{xy})_{x,y\in V}$. We shall refer to this random walk on $V$ simply as RW.

Next to these $r$-rates, we will also need a second-order version of them: for all distinct $x,y\in V$, 
	\begin{equation}\label{eq:rates-sxy}
	s_{xy}\eqdef \int_{[0,1]^2}\quadre{\beta_{xy}(\dd u,\dd v)+\beta_{yx}(\dd v,\dd u)}u\tonde{1-u}\fstop
\end{equation}

The following assumption will be enough to ensure uniqueness of invariant measures for our stochastic exchange models.
\begin{assumption}\label{ass:irr}For all $x,y\in V$, there exist $x_1,x_2\ldots, x_{\ell-1},x_\ell\in V$ satisfying
	\begin{equation}
		s_{xx_1}s_{x_1x_2}\cdots s_{x_{\ell-1}x_\ell}s_{x_\ell y}>0	\fstop
	\end{equation}
\end{assumption}
\begin{remark}\label{rem:irr} By confronting \eqref{eq:rates-rxy} with \eqref{eq:rates-sxy}, we have $s_{xy}\le r_{xy}$ for all distinct $x,y\in V$. Therefore,
	Assumption \ref{ass:irr}  implies irreducibility of RW and, thus, the existence of a unique positive invariant measure  $\pi=(\pi_x)_{x\in V}\in \cP(V)$ for RW, that is,
	\begin{equation}\label{eq:pi-invariance}
		 \sum_{y\neq x}\pi_y\,r_{yx}=\pi_x\sum_{y\neq x}r_{xy}\comma\qquad x \in V\fstop
	\end{equation}
\end{remark}

\begin{proposition}[Invariant measure]\label{pr:unique}Under Assumptions \ref{ass:beta} and \ref{ass:irr}, the stochastic exchange model on $\Omega$ described by \eqref{eq:gen-general} admits a unique invariant measure.	
\end{proposition}
Since  $\Omega$ is compact, the existence part of this proposition was already part of Proposition \ref{pr:invariance-polynomials}; the uniqueness part will be proved in Section \ref{sec:particle-system}, see  Corollary \ref{cor:unique-inv}.

\subsection{Main results}\label{sec:results}
In this section, Assumptions \ref{ass:beta} and \ref{ass:irr} are in force.

In view of Assumption \ref{ass:beta} and Proposition \ref{pr:invariance-polynomials}, $\cL$, restricted to the finite-dimensional space $\mathscr P_k$ of $k$-degree polynomials in $\eta$, is a well-defined linear operator. Hence, 	$-\cL\restr{\mathscr P_k}$ admits a decomposition in (possibly complex) eigenvalues and eigenfunctions. In particular, by \eqref{eq:rates-rxy}, when $k=1$, such eigenvalues coincide with those of the generator of RW with rates $(r_{xy})_{x,y\in V}$. Assumption \ref{ass:irr} and Remark \ref{rem:irr} ensure that only one of such eigenvalues is zero. Any other RW-eigenvalue $\lambda \in \C$ satisfies  (see, e.g., \cite[Theorem 2.1.7]{saloff1997lectures})
	\begin{equation}\label{eq:gap-1-bound}
		 {\rm gap}_{\rm RW}(\cL)\le {\rm Re}(\lambda)\comma
	\end{equation}
	where ${\rm gap}_{\rm RW}(\cL)>0$ is the RW's spectral gap, that is, 
	\begin{equation}\label{eq:gap-RW}
		{\rm gap}_{\rm RW}(\cL)\eqdef \inf\set{\cE_{\rm RW}(\psi): \psi\in \R^V\ \text{with}\ \ttnorm{\psi-\pi\psi}_{L^2(\pi)}=1}\comma 
	\end{equation}
	with $\cE_{\rm RW}$ being the RW's Dirichlet form: 
	\begin{equation}\label{eq:dir-form}
		 \cE_{\rm RW}(\psi)\eqdef \frac12\sum_{x,y\in V}\pi_{xy}\tonde{\psi(y)-\psi(x)}^2\comma\qquad \pi_{xy}\eqdef \frac{\pi_xr_{xy}+\pi_yr_{yx}}{2}\fstop
	\end{equation}

	 The bound in \eqref{eq:gap-1-bound} always holds true for eigenvalues of $-\cL\restr{\mathscr P_k}$, for $k=1$.   Our  main result provides a related bound for all $k\ge 2$. Its proof is the content of Section \ref{sec:proof}.
\begin{theorem}\label{th:main}
	Fix $k\ge 2$. Let $\lambda\in \C$ be an eigenvalue of $-\cL\restr{\mathscr P_k}$, but not of $\cL\restr{\mathscr P_{k-1}}$. Then, 
	\begin{equation}
	\gamma(\cL)\,{\rm gap}_{\rm RW}(\cL)\le 	{\rm Re}(\lambda)\comma
	\end{equation}
	where $\gamma(\cL)$ is defined as 
		\begin{equation}\label{eq:gamma-general}
		\gamma(\cL)\eqdef \min_{x,y\,:\,\pi_{xy}>0}\frac{\chi_{xy}+\sigma_{xy}}{\pi_{xy}}\comma
	\end{equation}with $\pi_{xy}$ given as in \eqref{eq:dir-form} (see also \eqref{eq:rates-rxy}), and
	\begin{align}\label{eq:chi-xy}
		\chi_{xy}=\chi_{yx}&\eqdef \int_{[0,1]^2}\quadre{\beta_{xy}(\dd u,\dd v)+\beta_{yx}(\dd v,\dd u)}\tonde{\pi_x\,u(1-u)+\pi_y\,v(1-v)}\comma
		\\
		\label{eq:sigma-xy}
		\sigma_{xy}=\sigma_{yx}&\eqdef \int_{[0,1]^2}\quadre{\beta_{xy}(\dd u,\dd v)+\beta_{yx}(\dd v,\dd u)}\tonde{\pi_x\,(1-u)-\pi_y\,(1-v)}^2\fstop
	\end{align}
	
\end{theorem}
\begin{remark}
	While $\chi_{xy}$ is proportional to $\pi_x+\pi_y$, $\sigma_{xy}$ is comparable to $(\pi_x+\pi_y)^2$. In many examples, this translates into $\sigma_{xy}\ll\chi_{xy}$  in the large-graph limit $|V|=n\to \infty$.
\end{remark}
\begin{remark}
	As a simple manipulation shows, 
the upper bound $\gamma(\cL)\le 2$ always holds.  Lower bounds for $\gamma(\cL)$ are less trivial to derive in full generality, but are easily checked case by case; see Section \ref{sec:examples}.
\end{remark}

The above theorem does not require any reversibility, nor a specific form of the invariant measure. At this level of generality, eigenvalues of the infinite-dimensional operator $-\cL$ do not necessarily exhaust its spectrum, which, in turn, may depend on the functional space on which $-\cL$ acts on.

Our next result fully exploits Theorem \ref{th:main}, showing that when the process is reversible with respect to a non-degenerate Dirichlet distribution, the 
$L^2$-spectrum coincides with the set of polynomial eigenvalues.  Consequently, the estimates in Theorem \ref{th:main} yield explicit bounds on the spectral gap.
\begin{theorem}[Reversible]\label{th:main-rev}  Let $\alpha=(\alpha_x)_{x\in V}$ be some positive weights and $\mu\sim {\rm Dir}(\alpha)$. Assume that  $\cL\restr{\mathscr P}$   is a symmetric operator on $L^2(\mu)$.
	Then, the spectrum of (the $L^2(\mu)$-graph-closure of)  $-\cL\restr{\mathscr P}$ is pure-point, real and nonnegative, and ${\rm gap}(\cL)$, its second  smallest eigenvalue, satisfies
	\begin{equation}\label{eq:main-gap-inequalities}
	\tonde{1\wedge	\gamma(\cL)}{\rm gap}_{\rm RW}(\cL)\le {\rm gap}(\cL)\le {\rm gap}_{\rm RW}(\cL)\fstop
	\end{equation}
	where ${\rm gap}_{\rm RW}(\cL)$ and $\gamma(\cL)$ are given in \eqref{eq:gap-RW} and \eqref{eq:gamma-general}, respectively. 
\end{theorem}
The proofs of the last two  theorems are deferred to Section \ref{sec:proof}. In Section \ref{sec:examples}, we apply them to the three models discussed in Theorems \ref{th:gap-beta}--\ref{th:gap-IEM}.
\begin{remark}\label{rem:gap-rw-1}
In this reversible setting,  ${\rm gap}_{\rm RW}$ and ${\rm gap}_1$ (cf.\  \eqref{eq:gap-RW} and \eqref{eq:gap-gap-k-def})  coincide.
\end{remark}

\begin{remark}The lower bound in \eqref{eq:main-gap-inequalities} remains valid under slightly more general assumptions, see Remark \ref{rem:general-assumptions-reversibile}.
\end{remark}

\subsubsection{Aldous' spectral gap identities}\label{sec:aldous-spectral-gap-identity}
Theorem \ref{th:main-rev} provides both upper and lower bounds for ${\rm gap}(\mathcal{L})$ solely in terms of\begin{equation}
	{\rm gap}_{\rm RW}(\mathcal{L})\qquad \text{and}\qquad 1 \wedge \gamma(\mathcal{L})\fstop 
\end{equation}In particular, whenever $\gamma(\mathcal{L}) \ge 1$,  the spectral gap of the stochastic exchange model coincides with that of the corresponding random walk on the same graph. This striking dimensionality reduction holding for any underlying graph $G$ is reminiscent of the celebrated Aldous' spectral gap conjecture---later resolved by Caputo, Liggett, and Richthammer \cite{caputo_proof_2010}---which established a similar identity between the spectral gaps of the interchange process, the symmetric exclusion process, and the associated random walk.

Somewhat surprisingly, such reductions have recently emerged in a number of structurally diverse models, see \cite{hermon_version_2019,quattropani2021mixing,salez2022universality,salez_spectral_2024,kim_sau_spectral_2023,kim_sau_one_2024}. Consequently, identifying conditions that ensure $\gamma(\mathcal{L}) \ge 1$ offers a practical criterion for verifying an analogue of Aldous' spectral gap identity within this broad class of stochastic exchange models reversible with respect to a Dirichlet distribution.
In this setting, the relatively explicit expression of $\gamma(\mathcal{L})$ given in \eqref{eq:gamma-general} allows us to derive, after a simple manipulation, the following effective criterion.

\begin{corollary}\label{cor:aldous-conj}
	In the context of Theorem \ref{th:main-rev}, we have ${\rm gap}(\cL)={\rm gap}_{\rm RW}(\cL)$ as soon as
	\begin{equation}
		\int_{[0,1]^2}\quadre{\beta_{xy}(\dd u,\dd v)+\beta_{yx}(\dd v,\dd u)}\tonde{1-u}\tttonde{2u-1}\ge 0\comma\quad \text{for all distinct}\  x, y \in V\fstop		
	\end{equation}

\end{corollary}

\section{Particle systems,  intertwinings, and hidden parameter models}\label{sec:preliminary}
In this section, we introduce three families of Markov processes, all closely related to the stochastic exchange models described in Section \ref{sec:SEM}: an interacting particle system, its time reversal, and the hidden parameter model. Here, we collect their most relevant properties, as well as duality and intertwining relations, of independent interest and all instrumental to the proof of our main results.
\subsection{Particle systems and forward intertwinings}
\label{sec:particle-system}

Particle systems arise naturally from the stochastic exchange models introduced in Section \ref{sec:SEM}, and turn out to be convenient to work with. Informally, they can be described as follows.

Let $(\eta(t))_{t\ge 0}$ be  a fixed realization of the process with generator $\cL$ in \eqref{eq:gen-general}. Recall that, for any time $t\ge 0$, $\eta(t)$ is a probability distribution on $V$.  Consider $k\ge 1$  particle positions $X_1(0),\ldots, X_k(0)$ on $V$, being i.i.d.\ and such that $X_i(0)\sim \eta(0)\in \Omega$. Given this initial condition, let particles move at each $\eta$-update: assuming that there is a first update  at time $t>0$ involving the vertices $x,y\in V$ and fractions $u,v\in [0,1]$, then set
\begin{equation}
	X_i(t)=\begin{dcases}
		x &\text{if}\ X_i(t^-)=x,\ \text{w/ prob.}\ u\ \vee \text{if}\ X_i(t^-)=y,\ \text{w/ prob.}\ 1-v\comma\\
		y &\text{if}\ X_i(t^-)=y,\ \text{w/ prob.}\ v\ \vee \text{if}\ X_i(t^-)=x,\ \text{w/ prob.}\ 1-u\comma	\\	X_i(t^-)&\text{else}\comma
	\end{dcases}
\end{equation}
and this independently over $i=1,\ldots, k$. This Markovian dynamics clearly generates, at any later time $t>0$, $k$ independent samples from $\eta(t)\in \Omega$. This is true for a given realization of $(\eta(t))_{t\ge 0}$, that is, in a quenched sense. The main point is that the annealed process obtained through an averaging over all updates introduces a dependence over the particles' evolution, although the particle system preserves the Markov property. We now formalize these ideas.

Fix an integer $k\ge 1$.
Recall from \eqref{eq:f-polynomial} the general form of a $k$-degree polynomial $f\in \mathscr P_k$ in the $\eta$-variables. 
Letting ($\N_0\eqdef \set{0,1,2,\ldots}$)
\begin{equation}
	\Xi_k\eqdef \set{\xi=(\xi_x)_{x\in V}\in \N_0^V: |\xi|=k }
\end{equation}
denote the space of unlabeled $k$-particle configurations on $V$, \eqref{eq:f-polynomial} becomes\footnote{We find convenient to work with both labeled and unlabeled versions of the particle systems, depending	on the context. With a slight abuse of notation,  we shall write observables as either $\psi(\xi)$ or $\psi(x_1,\ldots, x_k)$, always meaning that $(x_1,\ldots, x_k)\in V^k\mapsto \psi(x_1,\ldots, x_k)$ is symmetric and defined as $\psi(\delta_{x_1}+\ldots + \delta_{x_k})$.}
\begin{equation}\label{eq:f-psi}
	f(\eta)=k!\sum_{\xi\in \Xi_k} \psi(\xi)\,\prod_{x\in V}\frac{\eta_x^{\xi_x}}{\xi_x!}\comma\qquad \eta\in \Omega\fstop
\end{equation}
We shall encode this relation   between  $\psi\in \C^{\Xi_k}$ and $f\in \mathscr P_k$ as
\begin{equation}\label{eq:map-hat}
\psi\longmapsto \widehat \psi=f\fstop	
\end{equation}
In particular, this map is clearly linear and injective. Moreover, we have, for all $\psi\in \C^{\Xi_k}$, 
\begin{equation}\label{eq:int-relation}
	\cL\widehat \psi = \widehat{L\psi}\comma\end{equation} for some linear operator $L$, which we now identify as a particle 	system generator. In what follows, we shall refer to the identity in \eqref{eq:int-relation} as the forward intertwining relation (with \textquotedblleft forward\textquotedblright\ here used in opposition to the version introduced in Section \ref{sec:HPM} below)
	
\begin{proposition}\label{pr:particle}
	The  operator $L$ satisfying  \eqref{eq:int-relation} is a Markov generator acting on functions $\psi\in \R^{\Xi_k}$ as 
	\begin{equation}
		L\psi(\xi)=\sum_{x,y\in V}\sum_{\zeta\in \Xi_k} L_{xy}(\xi,\zeta)\set{\psi(\zeta)-\psi(\xi)}\comma\qquad \xi \in \Xi_k\comma
	\end{equation}
	where 
	\begin{equation}\label{eq:Lxy}
		L_{xy}(\xi,\zeta)=\int_{[0,1]^2}\beta_{xy}(\dd u,\dd v)\, M_{xy}^{uv}(\xi,\zeta)\comma
	\end{equation}
	and $M_{xy}^{uv}(\xi,\zeta)=0$ if $\xi_z\neq \zeta_z$ for some $z\neq x,y$, else
	\begin{equation}\label{eq:Mxyuv}
		M_{xy}^{uv}(\xi,\zeta)=	 \sum_{j_x=0}^{\xi_x\wedge\zeta_x} \binom{\xi_x}{j_x}\binom{\xi_y}{\zeta_x-j_x} u^{j_x}(1-u)^{\xi_x-j_x} v^{\xi_y-(\zeta_x-j_x)}(1-v)^{\zeta_x-j_x}\comma
	\end{equation}		
with 	the summation running over  $j_x\ge   (\xi_x+\xi_y)-(\xi_y+\zeta_y)$.
\end{proposition}
\begin{proof}
	First of all, we show that $L$ is an infinitesimal generator on $\R^{\Xi_k}$, that is,   
	\begin{equation}\label{eq:cond-generator}
		L1=0\qquad\text{and}\qquad \psi(\xi)=\min_{\zeta\in \Xi_k}\psi(\zeta)\ \Rightarrow L\psi(\xi)\ge 0\fstop
	\end{equation}
	Observe that  \eqref{eq:map-hat} maps constants into constants; more precisely,  $\widehat 1=1$. Hence, by $\cL1=0$ and \eqref{eq:int-relation}, we get  $0=\cL\widehat 1=\widehat{L1}$, and the first claim in \eqref{eq:cond-generator} follows by injectivity of the map in \eqref{eq:map-hat}. The second claim in \eqref{eq:cond-generator} is a  consequence of the first one and the intertwining relation in \eqref{eq:int-relation} with $\cL$.
	
	The rest of the proof for establishing the precise form of the rates is a computation, which we quickly 	sketch. For all $x,y\in V$, $u,v\in [0,1]$, and $\zeta \in \Xi_k$, we have
	\begin{equation}
		\prod_{z\in V} \frac{((\eta R_{xy}^{uv})_z)^{\zeta_z}}{\zeta_z!} = \sum_{\xi\in \Xi_k} \tonde{\prod_{z\in V}\frac{\eta_z^{\xi_z}}{\xi_z!}} M_{xy}^{uv}(\xi,\zeta)\comma
	\end{equation}
	where $M_{xy}^{uv}(\xi,\zeta)$ is given in \eqref{eq:Mxyuv}. 	Hence, $\cL \widehat \psi(\eta)$ equals $k!$ times
	\begin{align}
		& \sum_{x,y\in V}\sum_{\zeta\in \Xi_k} \psi(\zeta) \int_{[0,1]^2}\beta_{xy}(\dd u,\dd v)\bigg\{\prod_{z\in V}\frac{((\eta R_{xy}^{uv})_z)^{\zeta_z}}{\zeta_z!}-\prod_{z\in V}\frac{\eta_z^{\zeta_z}}{\zeta_z!}\bigg\}\\
		=\,& \sum_{x,y\in V} \sum_{\xi\in \Xi_k}\bigg\{\sum_{\substack{\zeta \in \Xi_k\\
					\zeta \neq \xi}}L_{xy}(\xi,\zeta)\, \psi(\zeta)-\int_{[0,1]^2}\beta_{xy}(\dd u,\dd v)\tonde{1-M_{xy}^{uv}(\xi,\xi)}\psi(\xi)\bigg\}\bigg(\prod_{z\in V}\frac{\eta_z^{\xi_z}}{\xi_z!}\bigg)\comma
	\end{align}
	and, by \eqref{eq:int-relation}, the expression in curly brackets on the right-hand side equals $L\psi(\xi)$.	
\end{proof}

In the next lemma, we show that, thanks to our Assumptions \ref{ass:beta} and \ref{ass:irr},  the particle system found in Proposition \ref{pr:particle} is irreducible on $\Xi_k$. 
\begin{lemma}\label{lem:irreducible-particle-system}
	$L$ describes an irreducible Markov chain on $\Xi_k$. 
\end{lemma}

\begin{proof}
	By Assumption \ref{ass:irr} and Remark \ref{rem:irr}, the claim is true for $k=1$; indeed, $r_{xy}=L_{xy}(\delta_x,\delta_y)+L_{yx}(\delta_x,\delta_y)$. For $k\ge 2$, it suffices to prove that, for all $x,y\in V$ satisfying $r_{xy}>0$, for all $\xi\in \Xi_k$ with $\xi_x\ge 1$ and $\ell_x\in \set{1,\ldots, \xi_x}$, there is a positive rate to jump either \textit{directly} from $\xi$ to $\xi-\ell_x\delta_x+\ell_x\delta_y$, or \textit{indirectly}, by passing  through the \textquotedblleft empty-arrival-vertex\textquotedblright\ configuration $\xi+\xi_y\delta_x-\xi_y\delta_y\in \Xi_k$. 
	
	Recall \eqref{eq:Lxy}--\eqref{eq:Mxyuv}. The rate of the first scenario is bounded below as follows
	\begin{align}
	&\tonde{L_{xy}+L_{yx}}(\xi,\xi-\ell_x\delta_x+\ell_x\delta_y)	\eqdef 	L_{xy}(\xi,\xi-\ell_x\delta_x+\ell_x\delta_y)+L_{yx}(\xi,\xi-\ell_x\delta_x+\ell_x\delta_y)\\
		&\quad\ge \int_{[0,1]^2}\quadre{\beta_{xy}(\dd u,\dd v)+\beta_{yx}(\dd v,\dd u)} u^{\xi_x}\tonde{1-u}^{\ell_x}v^{\xi_y}  \fstop
	\end{align}
	The rate of the first step of the second scenario is bounded below as follows
	\begin{align}
		(L_{xy}+L_{yx})(\xi,\xi+\xi_y\delta_x-\xi_y\delta_y)\ge \int_{[0,1]^2}\quadre{\beta_{xy}(\dd u,\dd v)+\beta_{yx}(\dd v,\dd u)}u^{\xi_x}\tonde{1-v}^{\xi_y}\fstop
	\end{align}
	The rate of the second step of the second scenario is readily estimated as done in the first scenario, with $\xi_y=0$ therein:
	\begin{align}
		&\tonde{L_{xy}+L_{yx}}(\xi+\xi_y\delta_x-\xi_y\delta_y,\xi-\ell_x\delta_x+\ell_x\delta_y)\\
		&\quad\ge \int_{[0,1]^2}\quadre{\beta_{xy}(\dd u,\dd v)+\beta_{yx}(\dd v,\dd u)} u^{\xi_x+\xi_y}\tonde{1-u}^{\ell_x+\xi_y}\fstop
	\end{align}
	
We are done, because,  if the right-hand side of the first display is zero, then Assumption \ref{ass:irr} ensures that those of the last two displays must be positive.\end{proof}
Recall that, by Assumption \ref{ass:beta} and Proposition \ref{pr:invariance-polynomials}, $\cL$ admits at least one invariant measure. 
If $\mu \in \cP(\Omega)$ denotes such a measure, then $\widehat \mu_k \in \cP(\Xi_k)$---defined via dual pairing  $\widehat \mu_k \psi = \mu \widehat \psi$ with functions $\psi \in \C^{\Xi_k}$---is invariant for $L$. Indeed,  $\mu \cL=0$ on $\mathscr P$ implies  $\widehat \mu_k L \psi = \mu \widehat{L\psi}=\mu \cL \widehat \psi=0$ for all $\psi \in \C^{\Xi_k}$. These considerations and Lemma \ref{lem:irreducible-particle-system} together prove the following result and, thus, Proposition \ref{pr:unique}.

\begin{corollary}\label{cor:unique-inv}Both $\cL$ and $L$ admit a unique invariant measure.
\end{corollary}
\begin{remark}
By definition, we have
	\begin{equation}
		\widehat \mu_k(\xi)= k!\int_\Omega \tonde{\prod_{x\in V}\frac{\eta_x^{\xi_x}}{\xi_x!}}\mu(\dd \eta)\comma\qquad \xi \in \Xi_k\fstop
\end{equation}
In particular, $\widehat \mu_1(\delta_x)=\pi_x$, $x\in V$, where $\pi=(\pi_x)_{x\in V}$ is given in Remark \ref{rem:irr}.
\end{remark}

\subsection{Annihilation and creation operators}
Fix an integer $k\ge 1$,  and introduce the (particle) annihilation operator $\mathfrak a: \C^{\Xi_{k-1}}\to \C^{\Xi_k}$ (conventionally, $\Xi_0\equiv \set{\emptyset}$ and $\C^{\Xi_0}\cong \C$), given by
\begin{equation}\label{eq:annhilation}
	\mathfrak a\psi(\xi)\eqdef \sum_{x\in V}\xi_x\,\psi(\xi-\delta_x)\comma\qquad \xi\in \Xi_k\comma \psi\in \C^{\Xi_{k-1}}\fstop
\end{equation}
Up to a normalization factor, $\mathfrak a \psi(\xi)$ is the average of the function $\psi$ after removing a particle uniformly at random from $\xi$. 

Our next result states that the particle system with generator $L$ commutes with this operation: removing a particle at the start and then evolving is the same (in law) as first evolving and then removing.
\begin{proposition}\label{pr:annihilation}
	$\mathfrak a L = L \mathfrak a$.
\end{proposition}
\begin{proof}[Proof (Sketch)]The proof goes either via a direct (but lengthy) computation, or via the observation that any subset of the $k$ \textquotedblleft quenched\textquotedblright\ particles $X_1(t),\ldots, X_k(t)$ discussed at the beginning of Section \ref{sec:particle-system} remains Markovian and evolves with the correct law. 
For an analogous	 statement, see, e.g., \cite{carinci_consistent_2019} or \cite[Appendix]{kim_sau_spectral_2023}.
\end{proof}
\begin{remark}\label{rem:subspace-invariant}
	Proposition \ref{pr:annihilation} implies that  $\mathfrak a(\C^{\Xi_{k-1}})\subset \C^{\Xi_k}$ is left invariant by $L$. 
\end{remark}

Recall that $\widehat \mu_k\in \cP(\Xi_k)$ stands for the unique invariant measure of $L$. Given $\mathfrak a:\C^{\Xi_{k-1}}\to \C^{\Xi_k}$, we obtain $\mathfrak a^\dagger : \C^{\Xi_k}\to \C^{\Xi_{k-1}}$ as its adjoint  \textquotedblleft creation\textquotedblright\ operator:
\begin{equation}
	\sum_{\zeta\in \Xi_{k-1}}\widehat \mu_{k-1}(\zeta)\,{\mathfrak a^\dagger \psi(\zeta)}\,\varphi(\zeta) = \sum_{\xi\in \Xi_k}\widehat \mu_k(\xi)\, \psi(\xi)\, \mathfrak a\varphi(\xi)\comma\qquad \varphi\in \C^{\Xi_k}\comma \psi\in \C^{\Xi_{k-1}}\fstop
\end{equation}
As an immediate consequence of this definition and Proposition \ref{pr:annihilation}, we have
\begin{equation}
	\mathfrak a^\dagger L^\dagger=L^\dagger\mathfrak a^\dagger\comma
\end{equation}
where $L^\dagger$ corresponds to the $\widehat \mu$-adjoint  generator of $L$. 	

\begin{remark} This time-reversed particle system will be relevant in Section \ref{sec:proof-control}, but 
we will not need its explicit jump rates.
\end{remark}
\subsection{Hidden parameter models and backward intertwinings}\label{sec:HPM}

Given the generator $\cL$ in \eqref{eq:gen-general}, $\mathscr L$ in \eqref{eq:gen-HPM} is the corresponding hidden parameter model's generator. Assumption \ref{ass:beta} and an argument as in the proof of Proposition \ref{pr:invariance-polynomials} guarantees that $\mathscr L$ is a well-defined operator on polynomials on $[0,1]^V$, and uniquely identifies a Feller process $(\theta(t))_{t\ge 0}$ therein.

\begin{remark}
Recall from \eqref{eq:HPM-update-matrix} that a general $\theta$-update reads as 
\begin{equation}
	(\theta_x,\theta_y)\longmapsto (U\theta_x+(1-U)\theta_y,(1-V)\theta_x+V\theta_y)\comma
\end{equation}
at some rate which may depend on $x,y\in V$ and $U,V\in [0,1]$, but not on $\theta\in [0,1]^V$. Hence,  the three basic properties listed below \eqref{eq:HPM-update} remain valid for all hidden parameter models described by $\mathscr L$ in \eqref{eq:gen-HPM}, under Assumptions \ref{ass:beta} and \ref{ass:irr}: mass is not necessarily conserved; the maximum principle holds; all invariant states consist of flat configurations only.
\end{remark}

The next proposition considerably generalizes the duality relations found in \cite{giardina_redig_tol_intertwining_2024} between $\mathscr L$ and $L$.
\begin{proposition}\label{pr:duality}
	Fix $k\ge 1$.
		Then, for all $b\in \R$, $\xi\in \Xi_k$ and $\theta\in [0,1]^V$, we have
	\begin{equation}\label{eq:duality}
		LD_b(\emparg,\theta)(\xi)=\mathscr L D_b(\xi,\emparg)(\theta)\comma\quad\text{with}\  D_b(\xi,\theta)\eqdef \prod_{x\in V}\tonde{\theta_x-b}^{\xi_x}\fstop
	\end{equation}
\end{proposition}
	\begin{proof}First, we prove the identity for $b=0$, i.e., for $D_0(\xi,\theta)=\prod_{z\in V}\theta_z^{\xi_z}$. Recalling the definition of \eqref{eq:gen-HPM}, we obtain, for all $x,y \in V$, $u,v\in [0,1]$, $\xi\in \Xi_k$ and $\theta \in [0,1]^V$, 
		\begin{equation}
		((R_{xy}^{uv}\theta)_z)^{\xi_z} = \begin{dcases}
			u\theta_x+(1-u)\theta_y &\text{if}\ z=x\\
			(1-v)\theta_x+v\theta_y &\text{if}\ z=y\\
			\theta_z &\text{else}\fstop
		\end{dcases}
		\end{equation}Hence, 
		\begin{equation}
		\prod_{z\in V}((R_{xy}^{uv}\theta)_z)^{\xi_z} = \sum_{\zeta\in \Xi_k} \tonde{\prod_{z\in V}\theta_z^{\xi_z}} M_{xy}^{uv}(\xi,\zeta)\comma
		\end{equation}
		where $M_{xy}^{uv}(\xi,\zeta)$ was given in \eqref{eq:Mxyuv}. Taking integrals with respect to $\beta_{xy}$ and summing over $x,y\in V$, we obtain the desired claim for $b=0$.

		For the general case, $b\neq 0$, first observe that
		\begin{equation}\label{eq:Db-D0}
			D_b(\xi,\theta)= \exp\tonde{-b\mathfrak a} D_0(\emparg,\theta)(\xi)\comma
		\end{equation}
		where $\mathfrak a$ is defined in \eqref{eq:annhilation} (note that the definition of $\mathfrak a$ does not depend on $k$). Indeed, $\mathfrak a =  \sum_{x\in V}\mathfrak a_x$, where $\mathfrak a_x\psi(\xi)\eqdef \xi_x\,\psi(\xi-\delta_x)$. Clearly,  $\mathfrak a_x\mathfrak a_y= \mathfrak a_y\mathfrak a_x$ for all $x,y\in V$. Hence, writing $D_0(\xi,\theta)= \prod_{z\in V}\theta_z^{\xi_z}\defeq \prod_{z\in V}d_0(\xi_z,\theta_z)$, we get
		\begin{align}
			\exp\tonde{-b\mathfrak a} D_0(\emparg,\theta)(\xi)&= \prod_{z\in V}\set{\exp\tonde{-b\mathfrak a_z} d_0(\emparg,\theta_z)(\xi_z)}\\
			&= \prod_{z\in V} \set{\sum_{\ell_z=0}^{\xi_z} \frac{(-b)^{\ell_z}}{\ell_z!} (\mathfrak a_z)^{\ell_z}d_0(\emparg,\theta_z)(\xi_z)}\\
			&=\prod_{z\in V}\set{\sum_{\ell_z=0}^{\xi_z} \frac{(-b)^{\ell_z}}{\ell_z!} \frac{\xi_z!}{(\xi_z-\ell_z)!}\, \theta_z^{\xi_z-\ell_z}}\\
			&= \prod_{z\in V} \tonde{\theta_z-b}^{\xi_z} = D_b(\xi,\theta)\comma
		\end{align}
		thus, proving \eqref{eq:Db-D0}.		The desired claim in \eqref{eq:duality} follows from this relation, the case $b=0$ which we already proved, and Proposition  \ref{pr:annihilation}; indeed, 
		\begin{align}
			L D_b&= L\exp\tonde{-b\mathfrak a} D_0 = \exp\tonde{-b\mathfrak a}LD_0 = \exp\tonde{-b\mathfrak a}\mathscr LD_0 = \mathscr L D_b\comma
		\end{align}
		where the last identity follows because $\exp\tonde{-b\mathfrak a}$ and $\mathscr L$ act on different variables.
	\end{proof}
	
	By the definition of adjoint generator $L^\dagger$ and the above proposition, we get the following \textquotedblleft dual\textquotedblright\ (or backward) intertwining relation between $\mathscr L$ and $L^\dagger$.
	\begin{corollary}\label{cor:intertwining-backward}Fix $k\ge 1$, and define the linear map from $\C^{\Xi_k}$ to polynomials on $[0,1]^V$
		\begin{equation}\label{eq:intertwining-backward}
			\widetilde \psi(\theta)\eqdef \sum_{\xi\in \Xi_k} \widehat \mu_k(\xi)\, \psi(\xi)\tonde{\prod_{x\in V}\theta_x^{\xi_x}}\comma\qquad \theta \in [0,1]^V\comma \psi\in \C^{\Xi_k}\fstop
		\end{equation}
		Then, for all $\psi\in \C^{\Xi_k}$, $\mathscr L\widetilde \psi = \widetilde{L^\dagger\psi}$.
	\end{corollary}

	\section{Proofs of main results	}\label{sec:proof}
In this section, we prove Theorems \ref{th:main} and \ref{th:main-rev}. The first two subsections develop the two key steps of the proof strategy described in Section \ref{sec:intro-proof}. The third subsection puts these two together, closing the proof of Theorem \ref{th:main}, while the final one is dedicated to the proof of Theorem \ref{th:main-rev}.
\subsection{Control on $k$-degree eigenfunctions}\label{sec:proof-control}
Recall the definitions of $\mathscr P_k$ and $\mathscr P_{k,\dagger}$ from \eqref{eq:Pk} and \eqref{eq:Pk-dagger}, and of the generators $\cL$ and $\mathscr L$ from \eqref{eq:gen-general} and \eqref{eq:gen-HPM}, respectively.
\begin{lemma}\label{lem:eigenfunction-control}
	Fix $k\ge 2$, and let $\lambda \in \C$ be an eigenvalue of $-\cL\restr{\mathscr P_k}$, but not of $-\cL\restr{\mathscr P_{k-1}}$.  Then, there exists a polynomial $g:[0,1]^V\to \C$ satisfying
	 \begin{equation}\label{eq:conditions-g-lemma}
	 	\mathscr L g=-\lambda g\comma\qquad  \sup_{\theta\in [0,1]^V}\abs{g(\theta)}>0\comma\qquad \sup_{\theta\in [0,1]^V} \frac{\abs{g(\theta)}}{{\rm Var}_\pi(\theta)^{k/2}}<\infty\fstop
	 \end{equation}
\end{lemma}
\begin{proof}Let $f\in \mathscr P_{k,\dagger}$ be an eigenfunction for $\lambda\in \C$.
	We divide the proof into three steps. In summary,  we first show that $f\in \mathfrak \cP_{k,\dagger}$  is linked to a \textquotedblleft true\textquotedblright\  eigenfunction $\psi\in \C^{\Xi_k}$ of the  $k$-particle system, that is, $\psi$ is not a function of only $\ell<k$ particles. Secondly, we prove that this $\psi$ corresponds to an analogous eigenfunction for the time-reversal particle system $\phi$ in $\ker(\mathfrak a^\dagger)$. The latter property is crucially exploited in the last step for constructing the hidden parameter model's eigenfunction $g$  in \eqref{eq:conditions-g-lemma}; see also Remark \ref{rem:g-Db}.
	
	\smallskip \noindent
	\textit{Step 1. There exists $0\neq \psi \in \C^{\Xi_k}$ satisfying}
	\begin{equation}\label{eq:step-1}
		\psi \notin \mathfrak a(\C^{\Xi_{k-1}})\comma\qquad f =\widehat \psi\comma\qquad  L\psi=-\lambda \psi\fstop
	\end{equation}
	
Let $\psi\in \C^{\Xi_k}$ be the unique function satisfying $f=\widehat \psi$.
Since $f\in \mathscr P_{k,\dagger}$, it cannot be that $\psi\in \mathfrak a(\C^{\Xi_{k-1}})$. If this were the case, we would have $\psi=\mathfrak a \varphi$ for some $\varphi \in \C^{\Xi_{k-1}}$, thus, 
\begin{equation}
	f=\widehat \psi=\widehat {\mathfrak a\varphi} = \widehat \varphi \in \mathscr P_{k-1}\comma
\end{equation}
leading to a contradiction, as $f\in \mathscr P_{k,\dagger}$. Finally, the third condition in \eqref{eq:step-1} is a consequence of $\cL f=-\lambda f$ and $f=\widehat \psi$.
	
	\smallskip \noindent
	\textit{Step 2. There exists $0\neq \phi\in \C^{\Xi_k}$ satisfying}
	\begin{equation}\label{eq:step-2}
		\phi \in \ker(\mathfrak a^\dagger)\comma\qquad L^\dagger \phi=-\lambda \phi\fstop
	\end{equation}
	
It is well-known that $L$ and $L^\dagger$---viewed as linear operators acting on the finite-dimensional space $\C^{\Xi_k}$---are isospectral (actually, as matrices,  $L^\dagger = J^{-1}L^TJ$, where $J={\rm diag}(\widehat \mu_k)$). 
 In particular, there must exist $0\neq \phi \in \C^{\Xi_k}$ satisfying 
\begin{equation}
	L^\dagger \phi=-\lambda \phi\fstop
\end{equation}
By Proposition \ref{pr:annihilation} (see also Remark \ref{rem:subspace-invariant}), $\mathfrak a(\C^{\Xi_{k-1}})\subset \C^{\Xi_k}$ is left invariant by $L$. Hence, $\mathfrak a(\C^{\Xi_{k-1}})$ is spanned by a family of (generalized) eigenfunctions, say \begin{equation}\epsilon_{1,1},\ldots, \epsilon_{1,i_1},\  \ldots\  \ldots\ ,\  {\epsilon_{m,1},  \ldots, \epsilon_{m,i_m}}\in \mathfrak a(\C^{\Xi_{k-1}})
	\comma
\end{equation} satisfying, for all $j=1,\ldots, m$ and some $\lambda_j\in \C$,
\begin{equation}\label{eq:generalized-eigenfunctions-chi}
L\epsilon_{j,1}=-\lambda_j\epsilon_{j,1}\comma\qquad	L\epsilon_{j,\ell}=-\lambda_j\epsilon_{j,\ell}+\epsilon_{j,\ell-1}\comma\quad \ell=2,\ldots, i_j\fstop
\end{equation}
Hence, for all $j=1,\ldots, m$, we obtain ($\ttscalar{\phi}{\epsilon}=\ttscalar{\phi}{\epsilon}_{L^2(\widehat \mu_k)}\eqdef\sum_{\xi\in \Xi_k}\widehat \mu_k(\xi)\, \phi^*(\xi)\,\epsilon(\xi)$)
\begin{equation}
	\lambda^* \ttscalar{\phi}{\epsilon_{j,1}}= \ttscalar{\lambda\phi}{\epsilon_{j,1}}=-\ttscalar{L^\dagger \phi}{\epsilon_{j,1}}=-\ttscalar{\phi}{L \epsilon_{j,1}}=\ttscalar{\phi}{\lambda_j\epsilon_{j,1}}=\ttscalar{\phi}{\epsilon_{j,1}}\lambda_j\fstop
\end{equation}
By the assumption of minimality, we have $\lambda \neq \lambda_j, \lambda_j^*$, $j=1,\ldots, m$, implying that $\ttscalar{\phi}{\epsilon_{j,1}}=0$. An iterative argument in $\ell=1,\ldots, i_j$ and using the second identity in \eqref{eq:generalized-eigenfunctions-chi} yields $\ttscalar{\phi}{\epsilon_{j,\ell}}=0$, for all $j=1,\ldots, m$ and $\ell=1,\ldots, i_j$. Since the generalized eigenfunctions $(\epsilon_{j,\ell})_{j,\ell}$ form a basis of $\mathfrak a(\C^{\Xi_{k-1}})$, we have $\ttscalar{\phi}{\varphi}=0$, for all $\varphi \in \mathfrak a(\C^{\Xi_{k-1}})$; by definition, this ensures that $\phi\in \ker(\mathfrak a^\dagger)\subset \C^{\Xi_k}$.
	
	\smallskip \noindent
	\textit{Step 3. $g\eqdef \widetilde \phi$  satisfies the conditions in \eqref{eq:conditions-g-lemma}.}

	By definition of backward intertwining in \eqref{eq:intertwining-backward} and Corollary \ref{cor:intertwining-backward}, we have
	\begin{equation}
		\mathscr L g(\theta)= \sum_{\xi\in \Xi_k}\widehat \mu_k(\xi)\, L^\dagger\phi(\xi)\tonde{\prod_{x\in V}\theta_x^{\xi_x}} = -\lambda g(\theta)\comma
	\end{equation}
	where the last step used the second condition in \eqref{eq:step-2}. Moreover, by the first condition in \eqref{eq:step-2} and the identity in \eqref{eq:Db-D0}, we get
	\begin{equation}\label{eq:g-Db}
		g(\theta)= \sum_{\xi\in \Xi_k} \widehat\mu_k(\xi)\, \phi(\xi)\, D_b(\xi,\theta)\comma\qquad \theta\in [0,1]^V\comma b\in \R\fstop
	\end{equation}
	Setting $b=\pi\theta$ on the right-hand side above, Cauchy-Schwarz inequality yields
	\begin{equation}
		\abs{g(\theta)}\le \ttnorm{\phi}_{L^2(\widehat \mu_k)}\ttnorm{D_{\pi\theta}(\emparg,\theta)}_{L^2(\widehat \mu_k)}\comma\qquad \theta \in [0,1]^V\fstop
	\end{equation}
	Clearly, $\ttnorm{\phi}_{L^2(\widehat \mu_k)}<\infty$; furthermore, we have
	\begin{align}
		\ttnorm{D_{\pi\theta}(\emparg,\theta)}_{L^2(\widehat \mu_k)}^2 &= \sum_{\xi\in \Xi_k}\widehat\mu_k(\xi)\tonde{\prod_{x\in V}\tonde{\theta_x-\pi\theta}^{2\xi_x}}\\
		&= \sum_{x_1,\ldots, x_k\in V} \mu(\eta_{x_1}\cdots \eta_{x_k})\tonde{\prod_{i=1}^k\tonde{\theta_{x_i}-\pi\theta}^2}\\
		&\le C\sum_{x_1,\ldots, x_k\in V} \mu(\eta_{x_1})\cdots \mu(\eta_{x_k})\tonde{\prod_{i=1}^k \tonde{\theta_{x_i}-\pi\theta}^2}\\
		&= C\, {{\rm Var}_\pi(\theta)}^k\comma
	\end{align}
	where  the second step follows from $\widehat \mu_k \varphi=\mu\widehat \varphi$ with $\varphi(x_1,\ldots,x_k)=\prod_{i=1}^k \tonde{\theta_{x_i}-\pi\theta}^2$,  the third step used that $\mu(\eta_x)=\pi_x$ and
	\begin{equation}
		\mu(\eta_{x_1}\cdots\eta_{x_k})\le C\,\mu(\eta_{x_1})\cdots \mu(\eta_{x_k})\comma\quad \text{with}\ C=\tonde{\min_{x\in V}\pi_x}^{-k}
		\comma
	\end{equation}
	 whereas the last step follows by definition of  ${\rm Var}_\pi(\theta)$; note that 	$C\in (0,\infty)$ by Assumption \ref{ass:irr}. This proves the third condition in \eqref{eq:conditions-g-lemma}.	 Since $\phi\in \ker(\mathfrak a^\dagger)$ and $\phi\neq 0$, verifying  $g\neq 0$ is immediate. This concludes the proof of the third step and, thus, of the lemma.
\end{proof}
\begin{remark}\label{rem:g-Db}
 In the above proof, \textit{Step 3} relies on the fact that the function
	\begin{equation}
	g(\theta) = \sum_{\xi \in \Xi_k} \widehat{\mu}_k(\xi)\, \phi(\xi)\, D_{\pi \theta}(\xi, \theta)\comma\qquad \theta \in [0,1]^V\comma
	\end{equation}
	in \eqref{eq:g-Db}
	is an eigenfunction of  $\mathscr L$. An essential ingredient here is Proposition \ref{pr:duality}, which asserts that  $D_b(\xi, \eta)$ satisfies the duality relation in \eqref{eq:duality} for any fixed constant $b \in \R$. However, this duality \textit{does not} hold when $b = \pi \theta$, since $b$ then depends on the variable $\theta$. This is precisely where the property $\phi \in \ker(\mathfrak a^\dagger)$ becomes crucial, ensuring that  the dependence on $b=\pi \theta$ in $g(\theta)$ is lost.
\end{remark}
\subsection{$L^2$-contraction of hidden parameter models}\label{sec:proof-L2}

The following lemma provides the key---and only---quantitative ingredient in our arguments. Recall the definitions around Theorem \ref{th:main}, and of the hidden parameter model generator $\mathscr L$ from \eqref{eq:gen-HPM}.
\begin{lemma}\label{lem:L2-contraction} Under Assumptions \ref{ass:beta} and \ref{ass:irr}, we have, for all $\theta\in [0,1]^V$,	
	\begin{equation}\label{eq:L2-norm-gamma-dir}
		\mathscr L\,{\rm Var}_\pi(\theta)\le -\gamma(\cL)\,\cE_{\rm RW}(\theta)\comma
	\end{equation}
	where ${\rm Var}_\pi(\theta)=\ttnorm{\theta-\pi\theta}_{L^2(\pi)}^2=\sum_{x\in V}\pi_x\tonde{\theta_x-\pi\theta}^2$, and $\gamma(\cL)$ is given in \eqref{eq:gamma-general}.
\end{lemma}
\begin{proof}In this proof, we simply write $\int \cdots=\int_{[0,1]^2}\cdots$, $\sum_{x,y}=\sum_{x,y\in V}$, $\ttnorm{\emparg}=\ttnorm{\emparg}_{L^2(\pi)}$.
	
	Since ${\rm Var}_\pi(\theta)=\norm{\theta}^2-(\pi\theta)^2$, 	we have
	$
	\mathscr L\,{\rm Var}_\pi(\theta) = \mathscr L\ttnorm{\theta}^2 - \mathscr L(\pi \theta)^2$. We start by computing this first term: 
	\begin{align}\label{eq:L2-norm-first}
		\mathscr L\ttnorm{\theta}^2 &= \sum_{x,y}\int\beta_{xy}(\dd u,\dd v)\,\ttset{\ttnorm{R_{xy}^{uv}\theta}^2-\ttnorm{\theta}^2}\fstop
	\end{align}
	Since 
	\begin{align}
		\ttnorm{R_{xy}^{uv}\theta}^2-\ttnorm{\theta}^2&= \theta_x^2\tonde{-\pi_x\tttonde{1-u^2}+\pi_y\tttonde{1-v}^2}+\theta_y^2\tonde{-\pi_y\tttonde{1-v^2}+\pi_x\tttonde{1-u}^2}\\
		&\qquad+2\theta_x\theta_y\tonde{\pi_x\,u\tonde{1-u}+\pi_y\,v\tonde{1-v}}\\
		&= \ttnorm{R_{yx}^{vu}\theta}^2-\ttnorm{\theta}^2\comma
	\end{align}
	the expression in \eqref{eq:L2-norm-first} reads as
	\begin{align}
		&\frac12\sum_{x,y}\theta_x^2 {\int\quadre{\beta_{xy}(\dd u,\dd v)+\beta_{yx}(\dd v,\dd u)}\tonde{-\pi_x\tttonde{1-u^2}+\pi_y\tttonde{1-v}^2}}\\
		+\,&\frac12\sum_{x,y}\theta_y^2 {\int\quadre{\beta_{xy}(\dd u,\dd v)+\beta_{yx}(\dd v,\dd u)}\tonde{-\pi_y\tttonde{1-v^2}+\pi_x\tttonde{1-u}^2}}\\
		+\,& \sum_{x,y}\theta_x\theta_y{
			\int\quadre{\beta_{xy}(\dd u,\dd v)+\beta_{yx}(\dd v,\dd u)}\tonde{\pi_x\,u(1-u)+\pi_y\,v(1-v)}}\fstop
	\end{align}
	Remark that $\chi_{xy}=\chi_{yx}$ in \eqref{eq:chi-xy} coincides with the expression in the last integral above.
	Then, after expanding $1-u^2=(1-u)+u(1-u)$, $(1-u)^2=(1-u)-u(1-u)$ (similarly for $v$), a part of the content of the first two integrals above equals $-\chi_{xy}$; for the remaining terms, recall  the definition of $r_{xy}$ from \eqref{eq:rates-rxy}. More precisely, we obtain
	\begin{align}
		\mathscr L\ttnorm{\theta}^2 &=-\frac12 \sum_{x,y}\chi_{xy}\tonde{\theta_x-\theta_y}^2+ \frac12\sum_{x,y}\tonde{\theta_x^2-\theta_y^2} \set{-\pi_x\,r_{xy}+\pi_y\,r_{yx}}\\
		&=-\frac12 \sum_{x,y}\chi_{xy}\tonde{\theta_x-\theta_y}^2\comma
	\end{align}
	where the last step is a consequence of the invariance of $\pi$, see \eqref{eq:pi-invariance}.

	We are left with the computation of
	\begin{align}\label{eq:L2-norm-second}
		\mathscr L (\pi\theta)^2 &=\sum_{x,y}\int\beta_{xy}(\dd u,\dd v)\set{(\pi R_{xy}^{uv}\theta)^2-(\pi\theta)^2}\fstop
	\end{align}
	By writing $\pi\theta =\sum_{z\in V}\pi_z\,\theta_z$ as $ (\pi_x\,\theta_x+\pi_y\,\theta_y)+\tau_{xy}=(\pi_x\,\theta_x+\pi_y\,\theta_y)+\sum_{z\neq x,y}\pi_z\,\theta_z$, we have
	\begin{align}
		&	(\pi R_{xy}^{uv}\theta)^2-(\pi\theta)^2 = (\pi R_{yx}^{vu}\theta)^2-(\pi\theta)^2 \\
		&= (\pi_x\tonde{u\,\theta_x+(1-u)\,\theta_y}+\pi_y\tonde{(1-v)\theta_x+v\theta_y} + \tau_{xy})^2-(\pi_x\,\theta_x+\pi_y\,\theta_y+\tau_{xy})^2\\
		&= \pi_x^2\tonde{(u^2-1)\,\theta_x^2+(1-u)^2\,\theta_y^2+2u(1-u)\,\theta_x\theta_y}\\
		&\quad+\pi_y^2\tonde{(1-v)^2\,\theta_x^2+(v^2-1)\,\theta_y^2+2 v(1-v)\,\theta_x\theta_y}\\
		&\quad+ 2\pi_x\pi_y\tonde{u(1-v)\,\theta_x^2+(1-u)v\,\theta_y^2+\tonde{2uv -u-v}\theta_x\theta_y}\\
		&\quad -2\tau_{xy}\tonde{\theta_x-\theta_y	}\tonde{\pi_x\,(1-u)-\pi_y\,(1-v)}\fstop
	\end{align}
	As a consequence, the expression in \eqref{eq:L2-norm-second} equals
	\begin{align}
		& \frac12\sum_{x,y}\theta_x^2{\int \quadre{\beta_{xy}(\dd u,\dd v)+\beta_{yx}(\dd v,\dd u)}\tonde{\pi_x^2\,(u^2-1) + \pi_y^2\,(1-v)^2 + 2\pi_x\pi_y\, u(1-v)}}\\
		+\,&\frac12\sum_{x,y}\theta_y^2{\int \quadre{\beta_{xy}(\dd u,\dd v)+\beta_{yx}(\dd v,\dd u)}\tonde{\pi_y^2\,(v^2-1)+\pi_x^2\,(1-u)^2+2\pi_x\pi_y\,v(1-u)}}\\
		+\,&\sum_{x,y}\theta_x\theta_y
		{\int \quadre{\beta_{xy}(\dd u,\dd v)+\beta_{yx}(\dd v,\dd u)} \tonde{\pi_x^2\,u(1-u)+\pi_y^2\, v(1-v)+\pi_x\pi_y\,(2uv-u-v)}}\\
		-\,&\sum_{x,y}\tau_{xy}\tonde{\theta_x-\theta_y}\set{\pi_x\, r_{xy}-\pi_y\,r_{yx}}\fstop	\end{align}				
	Then, by expanding $u^2-1=(1-u)^2-2(1-u)$, $u(1-v)=-(1-u)(1-v)+(1-v)$, $u(1-u)=-(1-u)^2+(1-u)$, $2uv-u-v=2(1-u)(1-v)-(1-u)-(1-v)$,	 etc., we  get
	\begin{align}
		\mathscr L(\pi\theta)^2&= \frac12\sum_{x,y}\sigma_{xy}\tonde{\theta_x-\theta_y}^2\\
		&-\sum_{x,y}\pi_x\,\theta_x^2\set{\pi_x\,r_{xy}-\pi_y\,r_{yx}} + \sum_{x,y}\pi_y\,\theta_y^2\set{\pi_x\,r_{xy}-\pi_y\,r_{yx}}\\
		&+\sum_{x,y}\theta_x\theta_y\tonde{\pi_x-\pi_y}\set{\pi_x\,r_{xy}-\pi_y\,r_{yx}}-\sum_{x,y}\tau_{xy}\tonde{\theta_x-\theta_y}\set{\pi_x\, r_{xy}-\pi_y\,r_{yx}}\\
		&=\frac12\sum_{x,y}\sigma_{xy}\tonde{\theta_x-\theta_y}^2 - \sum_{x,y}\tonde{\pi_x\theta_x+\pi_y\theta_y+\tau_{xy}}\tonde{\theta_x-\theta_y} \set{\pi_x\,r_{xy}-\pi_y\,r_{yx}}\\
		&=\frac12\sum_{x,y}\sigma_{xy}\tonde{\theta_x-\theta_y}^2\comma
	\end{align}
	where the last step used that $\pi_x\theta_x+\pi_y\theta_y+\tau_{xy}=\pi\theta$ does not depend on $x,y\in V$, and the invariance of $\pi$, i.e.,	 \eqref{eq:pi-invariance}. 
	
	Altogether, we obtained that $	\mathscr L\,{\rm Var}_\pi(\theta)=\mathscr L\ttnorm{\theta}^2-\mathscr L(\pi\theta)^2$ equals
	\begin{align}\label{eq:L2-norm-last}
		-\frac12\sum_{x,y}\tonde{\chi_{xy}+\sigma_{xy}}\tonde{\theta_x-\theta_y}^2 =-\frac12\sum_{x,y\,:\,\pi_{xy}>0}\frac{\chi_{xy}+\sigma_{xy}}{\pi_{xy}}\,\pi_{xy}\tonde{\theta_x-\theta_y}^2\comma
	\end{align}
	and, by \eqref{eq:dir-form} and \eqref{eq:gamma-general}, the desired claim follows.
\end{proof}
\begin{remark}
	Recall from \eqref{eq:gamma-general} that $\gamma(\cL)$ is defined as a minimum. If this coincides with the maximum of the same quantity, then the inequality in \eqref{eq:L2-norm-gamma-dir} becomes an identity.
\end{remark}

\subsection{Proof of Theorem \ref{th:main}}

The results of Lemmas \ref{lem:eigenfunction-control} and \ref{lem:L2-contraction} provide all we need to prove Theorem \ref{th:main} following the strategy outlined in Section \ref{sec:intro-proof}.
\begin{proof}[Proof of Theorem \ref{th:main}]
	Let  $\lambda\in \C$ and $f\in \mathscr P_{k,\dagger}$, $k\ge 2$, be an	 eigenpair of $-\cL$ as in Lemma \ref{lem:eigenfunction-control}. Recall that $(\mathscr S_t)_{t\ge 0}$ stands for the Markov semigroup of the hidden parameter model $(\theta(t))_{t\ge 0}$, whose generator $\mathscr L$ was given in \eqref{eq:gen-HPM}. 
	
	  Lemma \ref{lem:eigenfunction-control} implies that there exists $g:[0,1]^V\to \C$, $g\neq 0$, satisfying, for some $C=C(g)>0$ and for all  $t\ge 0$ and $\theta \in [0,1]^V$,
	\begin{equation}\label{eq:1}
		\ttabs{e^{-\lambda t}g(\theta)}=\ttabs{\mathscr S_t g(\theta)} \le \mathscr S_t\ttabs{g}(\theta)\le C\,		 (\mathscr S_t{\rm Var}_\pi)(\theta)\comma
	\end{equation}
	where, since $k\ge 2$, we used that ${\rm Var}_\pi(\theta)^{k/2}\le {\rm Var}_\pi(\theta)$.
	
	Lemma \ref{lem:L2-contraction} yields, for  all $t\ge 0$ and $\theta\in [0,1]^V$, 
	\begin{equation}
		\frac{\dd}{\dd t}(\mathscr S_t{\rm Var}_\pi)(\theta)=(\mathscr S_t \mathscr L{\rm Var}_\pi)(\theta)\le -\gamma(\cL)\,(\mathscr S_t\cE)(\theta)\le -\gamma(\cL)\, {\rm gap}_{\rm RW}(\cL)\, (\mathscr S_t{\rm Var}_\pi)(\theta)\comma
	\end{equation}
	where the last inequality follows by the definition of ${\rm gap}_{\rm RW}(\cL)$ in \eqref{eq:gap-RW}. By Gr\"onwall's inequality, we obtain, for all $t\ge0$ and $\theta\in [0,1]^V$,
	\begin{equation}\label{eq:2}
		(\mathscr S_t{\rm Var}_\pi)(\theta)\le e^{-\gamma(\cL)\,{\rm gap}_{\rm RW}(\cL)\,t}\,{\rm Var}_\pi(\theta)\le e^{-\gamma(\cL)\,{\rm gap}_{\rm RW}(\cL)\,t}\fstop
	\end{equation}
By combining \eqref{eq:1}, \eqref{eq:2} and $g\neq 0$, taking the $\log$ and sending $t\to \infty$, we obtain the desired estimate.	
\end{proof}

\subsection{Proof of Theorem \ref{th:main-rev}}
Fix $\alpha$ and $\mu\sim {\rm Dir}(\alpha)$.
 We divide the proof of Theorem \ref{th:main-rev} into two parts: one for the main claim and the lower bound in \eqref{eq:main-gap-inequalities}, one for the upper bound in \eqref{eq:main-gap-inequalities}.

\begin{proof}[Proof of the lower bound in Theorem \ref{th:main-rev}] 
	Since $\mu$ is absolutely continuous with respect to the uniform measure on $\Omega$, $\cL$ in \eqref{eq:gen-general}, viewed as an operator on  polynomials in $\mathscr P$, is well-defined as an operator on equivalence classes in $L^2(\mu)$. Since $\cL$  leaves $\mathscr P_k$, $k\ge 0$, invariant, and is symmetric on the the dense subspace $\mathscr P\subset L^2(\mu)$, there exists a c.o.n.s.\ 
	\begin{equation}\label{eq:cons}
		f_{kj}\in \mathscr P_{k,\dagger}\comma\qquad k\in \N_0\comma j=1,\ldots, N_k\eqdef|\Xi_k|-|\Xi_{k-1}|\comma	
	\end{equation}in $L^2(\mu)$ of polynomial eigenfunctions for $-\cL$, all associated to real, nonnegative eigenvalues. Hence,  $-\cL$ is essentially self-adjoint \cite[Theorem VIII.3]{reed_simon_methods_I_1972}, and the spectral theorem applies to its unique closed extension, ensuring a pure point spectrum. Theorem \ref{th:main} yields the lower bound in \eqref{eq:main-gap-inequalities}.
\end{proof}
\begin{proof}[Proof of the upper bound in Theorem \ref{th:main-rev}]
	 It suffices to prove that the lowest nonzero eigenvalue of $-\cL\restr{\mathscr P_1}$---and, thus, of $-L$ on $\R^{\Xi_1}$---equals ${\rm gap}_{\rm RW}(\cL)$ as defined in \eqref{eq:gap-RW}. For this purpose, we show that $L$ is self-adjoint on $L^2(\widehat \mu_1)\cong L^2(\pi)$,  $\pi_x=\widehat \mu_1(\delta_x)$, $x\in V$.
	
	Recalling the c.o.n.s.\ of (real-valued) polynomial eigenfunctions in \eqref{eq:cons}, we have
	\begin{equation}\label{eq:cons2}
		\ttscalar{f_{kj}}{f_{\ell i}}_{L^2(\mu)}=\car_{k=\ell,j=i}\comma\quad f_{kj}=\widehat \psi_{kj}\comma 
		\quad L \psi_{kj}=-\lambda_{kj}\psi_{kj}\fstop
	\end{equation}
	for some $\psi_{kj}\in \R^{\Xi_k}\setminus \mathfrak a(\R^{\Xi_{k-1}})$ and $\lambda_{kj}\ge0$.
	Clearly,  we have $f_{01}=1=\widehat 1$ and, by the first two identities in \eqref{eq:cons2} with $k\ge 1$ and $\ell=0$,  
	\begin{equation}\label{eq:cons3}\mu f_{kj}=\widehat \mu_k\psi_{kj}=0
	\fstop
	\end{equation} By the same identities in \eqref{eq:cons2} with $k=\ell=1$, we obtain
	\begin{align}
		\car_{j=i}= \ttscalar{f_{1j}}{f_{1i}}_{L^2(\mu)}&= \sum_{x,y\in V} \psi_{1j}(x)\,\psi_{1i}(y)\, \mu(\eta_x\eta_y)\\
		&=\frac1{|\alpha|+1}\ttscalar{\psi_{1j}}{\psi_{1i}}_{L^2(\pi)} + \frac{|\alpha|}{|\alpha|+1}\tonde{\pi\psi_{1j}}\tonde{\pi\psi_{1i}}\\
		&= \frac1{|\alpha|+1}\ttscalar{\psi_{1j}}{\psi_{1i}}_{L^2(\pi)}\comma
	\end{align}
	 where for the third step we used that $\mu\sim {\rm Dir}(\alpha)$ fulfills $\mu(\eta_x\eta_y)= \frac1{|\alpha|+1}\ttset{ \car_{x=y}\pi_x+ |\alpha|\pi_x\pi_y}$, while \eqref{eq:cons3} for the last one. Hence, this and \eqref{eq:cons2} ensure that $\mathfrak a 1,\psi_{11},\ldots, \psi_{1N_1}\in \R^{\Xi_1}$ forms an $L^2(\pi)$-orthogonal basis of eigenfunctions for $-L$. This proves the  claim.	
\end{proof}
\begin{remark}
	Along the same lines of the proof above, it is not difficult to check that if $\cL$ admits $\mu\sim {\rm Dir}(\alpha)$ as a reversible measure, then $L$ is self-adjoint on $L^2(\widehat \mu_k)$, not only for $k=1$, but for all $k\ge 1$.
\end{remark}

\begin{remark}\label{rem:general-assumptions-reversibile}
	The proof of the lower bound in Theorem \ref{th:main-rev} actually remains valid also when $\mu$ is not a Dirichlet measure, but merely absolutely continuous on $\Omega$.	
\end{remark}

\section{Examples. Proofs of Theorems \ref{th:gap-beta}--\ref{th:gap-IEM}}\label{sec:examples}
We now discuss in more detail a few examples of reversible stochastic exchange model	on an undirected weighted connected finite graph $G=(V,(c_{xy})_{x,y\in V})$, where $c_{xy}=c_{yx}\ge 0$ are symmetric conductances. Further,  we fix some positive weights $\alpha=(\alpha_x)_{x\in V}$ attached to vertices. Given all this, we think of the rates  $\beta_{xy}$ and $\beta_{yx}$  as being proportional to $c_{xy}\ge 0$ and depending on $\alpha_x, \alpha_y>0$. 

For all the models that we now consider,   Assumptions \ref{ass:beta} and \ref{ass:irr} always hold true, and  ${\rm Dir}(\alpha)$ in \eqref{eq:mu-dirichlet} is their unique reversible measure; thus, Theorem \ref{th:main-rev} and Corollary \ref{cor:aldous-conj} apply; see also Remark \ref{rem:gap-rw-1}. 
Before entering their details, we mention that all computations below, although somewhat lengthy, are simple, merely based on the following standard identities involving the Gamma function: for all $a,b>0$, 
\begin{equation}
	\int_0^1 u^{a-1}\tonde{1-u}^{b-1}\dd u = \frac{\Gamma(a)\Gamma(b)}{\Gamma(a+b)}\comma\qquad \Gamma(a+1)=a\Gamma(a)\fstop
\end{equation}
\subsection{The Kipnis-Marchioro-Presutti model} Define, 	for all distinct	 $x,y\in V$,
\begin{equation}
	\beta_{xy}(\dd u,\dd v) = \frac{c_{xy}}{2\,{\rm B}(\alpha_x,\alpha_y)}  u^{\alpha_x-1}\tonde{1-u}^{\alpha_y-1}\dd u\, \delta_{1-u}(\dd v)\comma
\end{equation}
where ${\rm B}(\alpha_x,\alpha_y)={\rm B}(\alpha_y,\alpha_x)$ is the Beta function (i.e., the normalization constant defined in \eqref{eq:mu-dirichlet} with $V=\set{x,y}$). These rates correspond to an edge heat-bath dynamics with respect to the Dirichlet distribution, possibly with edge and vertex inhomogeneities. Therefore, reversibility readily follows.

With this choice, the corresponding $r$-rates read as (cf.\ \eqref{eq:rates-rxy})
\begin{equation}
	r_{xy}= \frac{c_{xy}}{2\,{\rm B}(\alpha_x,\alpha_y)} \set{\int_0^1 u^{\alpha_x-1}\tonde{1-u}^{\alpha_y}\dd u + \int_0^1 v^{\alpha_y}\tonde{1-v}^{\alpha_x-1}\dd v }= c_{xy}\,\frac{\alpha_y}{\alpha_x+\alpha_y}\fstop
\end{equation}
Clearly, detailed balance for $(r_{xy})_{x,y\in V}$ holds with  $\pi=(\pi_x)_{x\in V}$ given by	
\begin{equation}\label{eq:pi-alpha}
	\pi_x=\frac{\alpha_x}{|\alpha|}\comma\quad \text{with}\ |\alpha|=\sum_{x\in V}\alpha_x\fstop
\end{equation} For what concerns the values of $\chi_{xy}$ and $\sigma_{xy}$ introduced in \eqref{eq:chi-xy}--\eqref{eq:sigma-xy}, we have
\begin{align}
	\chi_{xy}&= \frac{c_{xy}}{2\,{\rm B}(\alpha_x,\alpha_y)}\set{\frac{\alpha_x}{|\alpha|}\int_0^1 u^{\alpha_x}\tonde{1-u}^{\alpha_y}\dd u+\frac{\alpha_y}{|\alpha|}\int_0^1 v^{\alpha_y}\tonde{1-v}^{\alpha_x}\dd v
	}\\
	&=  \frac{c_{xy}}{|\alpha|}\,\frac{\alpha_x\alpha_y}{\alpha_x+\alpha_y+1}\comma
\end{align}
and
\begin{align}
	\sigma_{xy}&=\frac{c_{xy}}{2\,{\rm B}(\alpha_x,\alpha_y)} \set{\int_0^1 u^{\alpha_x-1}\tonde{1-u}^{\alpha_y-1}\tonde{\frac{\alpha_x}{|\alpha|}\tonde{1-u}-\frac{\alpha_y}{|\alpha|}u}^2\dd u }\\
	&+\frac{c_{xy}}{2\,{\rm B}(\alpha_x,\alpha_y)} \set{\int_0^1 v^{\alpha_y-1}\tonde{1-v}^{\alpha_x-1}\tonde{\frac{\alpha_y}{|\alpha|}\tonde{1-v}-\frac{\alpha_x}{|\alpha|}v}^2\dd v}\\
	&= \frac{c_{xy}}{|\alpha|^2}\frac{\alpha_x\alpha_y}{\alpha_x+\alpha_y+1}\fstop
\end{align}
In conclusion, substituting these values into \eqref{eq:gamma-general} yields
\begin{equation}
	\gamma_{\rm KMP}(G,\alpha)=\min_{x,y \,:\,c_{xy}>0} \frac{\alpha_x+\alpha_y}{\alpha_x+\alpha_y+1}\tonde{1+\frac1{|\alpha|}}= \frac{\alpha_{2,\rm min}}{\alpha_{2,\rm min}+1}\tonde{1+\frac1{|\alpha|}}\comma
\end{equation}
where $\alpha_{2,\rm min}$ is given in \eqref{eq:alpha-2-min}. This corresponds to the main claim in Theorem \ref{th:gap-beta}.

\subsubsection{Sharpness for ${\rm KMP}$}\label{sec:KMP-sharpness}
		When $\alpha_x\equiv \hat \alpha>0$ is constant, for the KMP model we recover
	\begin{equation}
		\gamma_{\rm KMP}(G,\alpha)= \frac{2\hat \alpha+\frac2n}{2\hat \alpha+1}\comma\qquad r_{xy}=\frac{c_{xy}}{2}\fstop
	\end{equation}

On the complete graph $G=K_n$ with $c_{xy}\equiv \frac1n$, it is immediate to check that ${\rm gap}_1(G,\alpha)=\frac12$. Hence, 
Theorem \ref{th:main-rev} yields the following lower bound for ${\rm gap}(G,\alpha)$:
\begin{equation}
	\gamma_{\rm KMP}(G,\alpha)\, {\rm gap}_1(G,\alpha)=\frac12 \frac{2\hat \alpha+\frac2n}{2\hat \alpha+1}\fstop
\end{equation} 
According to \cite[Eq.\ (10)]{caputo_mixing_2019}, this is precisely the spectral gap of the corresponding ${\rm KMP}$ model. Finally, recall that, in this homogeneous mean-field setting, the KMP model satisfies ${\rm gap}(G,\alpha)={\rm gap}_2(G,\alpha)$; see, e.g., \cite{caputo_kac2008}. 

On the segment with homogeneous $\alpha\equiv \hat \alpha \ge 1$, \cite[Theorem 2]{caputo_mixing_2019} proves that the upper bound in Theorem \ref{th:main-rev} is, in fact, an identity. See  \cite[Theorem 3]{labbe_petit_hydrodynamic_2025} for an analogous result on the segment with \begin{equation}\alpha_x=\hat\alpha\tonde{\frac{1+\sigma}{1-\sigma}}^{x-1}\comma\qquad x\in V=\set{1,\ldots, n}\comma \hat\alpha\ge 1\comma \sigma\in (0,1)\fstop
	\end{equation}

\subsection{Harmonic process} Here, we have
\begin{equation}
	\beta_{xy}(\dd u,\dd v)=c_{xy}\,	u^{\alpha_x-1}\tonde{1-u}^{-1}\dd u\otimes \delta_1(\dd v)\fstop
\end{equation}
For reversibility, see, e.g., \cite[Appendix A]{franceschini_frassek_giardina_integrable_2023} for the case $\alpha\equiv {\rm const.}$; the general case is analogously verified.

With this choice, we obtain
$	r_{xy}=c_{xy}\,\frac1{\alpha_x}$, for all $x,y\in V$,
and, clearly, these rates satisfy detailed balance with respect to $\pi$ given in \eqref{eq:pi-alpha}.
Further, we obtain
\begin{align}
	\chi_{xy}=c_{xy}\set{\frac{\alpha_x}{|\alpha|}\int_0^1 u^{\alpha_x}\, \dd u + \frac{\alpha_y}{|\alpha|}\int_0^1 v^{\alpha_y}\ \dd v} = \frac{c_{xy}}{|\alpha|}\set{\frac{\alpha_x}{\alpha_x+1}+\frac{\alpha_y}{\alpha_y+1}} \comma
\end{align}
\begin{align}
	\sigma_{xy}=\frac{c_{xy}}{|\alpha|^2}\set{\frac{\alpha_x}{\alpha_x+1}+\frac{\alpha_y}{\alpha_y+1}}\comma
\end{align}
from which we get
\begin{equation}
	\gamma_{\rm HP}(G,\alpha)=\min_{x,y\,:\,c_{xy}>0} \set{\frac{\alpha_x}{\alpha_x+1}+\frac{\alpha_y}{\alpha_y+1}}\tonde{1+\frac1{|\alpha|}}\comma
\end{equation}
and, thus, the main conclusion of Theorem \ref{th:gap-HP}.
\subsubsection{Sharpness for ${\rm HP}$}\label{sec:HP-sharpness}
	We have $\gamma_{\rm HP}(G,\alpha)\ge 1$, whenever $\alpha_{\rm min}\ge 1$, establishing the spectral gap identity ${\rm gap}(G,\alpha)={\rm gap}_1(G,\alpha)$; the criterion in Corollary \ref{cor:aldous-conj} yields the same result.
	
	We now argue that the lower bound in Theorem \ref{th:gap-HP} is attained on the homogeneous complete graph, that is, when  $c_{xy}\equiv \frac1n$ and $\alpha\equiv \hat \alpha>0$. Indeed, in this case, we have ${\rm gap}_1(G,\alpha)={\hat\alpha}^{-1}$ and
	\begin{equation}
		 \gamma_{\rm HP}(G,\alpha)\,{\rm gap}_1(G,\alpha)
		 =\frac{2}{\hat \alpha}\frac{\hat\alpha+\frac{1}{n}}{\hat\alpha+1}\fstop
	\end{equation}
Solving $\cL f_{a}=-\lambda_2 f_{a}$ for the degree-two polynomial $f_{a}(\eta)=\sum_{x\in V}\eta_x^2-a$ with respect to $a\in \R$ and $\lambda >0$, a simple computation yields
	\begin{equation}
	 a= \frac1n\frac{\hat\alpha+1}{\hat\alpha +\frac1n}\comma\qquad 	\lambda_2 =\frac{2}{\hat\alpha}\frac{\hat \alpha+\frac1n}{\hat\alpha+1}= \gamma_{\rm HP}(G,\alpha)\, {\rm gap}_1(G,\alpha)\fstop
	\end{equation} 
	This proves the desired claim. Finally, note that  $\lambda_2< {\rm gap}_1(G,\alpha)$ if and only if $\hat \alpha<1-\frac2n $.

\subsection{Immediate exchange models} Here, fix $\kappa\in (0,\alpha_{\rm min})$, and consider
\begin{equation}
	\beta_{xy}(\dd u,\dd v)=\frac{c_{xy}}{2\,{\rm B}(\alpha_x-\kappa,\kappa)\,{\rm B}(\alpha_y-\kappa,\kappa)}\, u^{\alpha_x-\kappa-1}\tonde{1-u}^{\kappa-1} v^{\alpha_y-\kappa-1}\tonde{1-v}^{\kappa-1}\, \dd u\dd v\fstop
\end{equation}
For the proof of reversibility of the measure ${\rm Dir}(\alpha)$, see \cite{ginkel_redig_sau_duality_2016,redig_generalized_2017}.
Moreover, we readily obtain  the rates
$
	r_{xy}={c_{xy}\,\frac{\kappa}{\alpha_x}}$, $x,y \in V$,
for which, again, the measure in \eqref{eq:pi-alpha} is reversible. Further, we compute
\begin{align}
	\chi_{xy}&= \frac{c_{xy}}{2\,{\rm B}(\alpha_x-\kappa,\kappa)}	{\frac{\alpha_x}{|\alpha|}\int_0^1 u^{\alpha_x-\kappa}(1-u)^{\kappa}\dd u} +\frac{c_{xy}}{2\,{\rm B}(\alpha_y-\kappa,\kappa)}\frac{\alpha_y}{|\alpha|}\int_0^1 v^{\alpha_y-\kappa}(1-v)^\kappa\dd v\\
	&= \frac{c_{xy}\,\kappa}{|\alpha|}\set{\frac{{\alpha_x-\kappa}}{\alpha_x+1}+\frac{{\alpha_y-\kappa}}{\alpha_y+1}}\comma
\end{align}
and
\begin{align}
	\sigma_{xy}&= \frac{c_{xy}\,\kappa}{|\alpha|^2}\set{\frac{\alpha_x-\kappa}{\alpha_x+1}+\frac{\alpha_y-\kappa}{\alpha_y+1}}\fstop
\end{align}
As a consequence, this yields the desired claim of Theorem \ref{th:gap-IEM}, that is,
\begin{equation}
	\gamma_{\rm IEM}(G,\alpha,\kappa) = \min_{x,y\,:\,c_{xy}>0} \set{\frac{\alpha_x-\kappa}{\alpha_x+1}+\frac{\alpha_y-\kappa}{\alpha_y+1}}\tonde{1+\frac1{|\alpha|}}\fstop
\end{equation}
\subsubsection{Sharpness for ${\rm IEM}$}\label{sec:IEM-sharpness}Recall that we required $\kappa<\alpha_{\rm min}$; further, 
	$\gamma_{\rm IEM}(G,\alpha,\kappa)\ge 1$ if $\alpha_{\rm min}\ge 1+2\kappa$. Hence, under these conditions, the upper bound in Theorem \ref{th:gap-IEM} saturates to an identity. Corollary \ref{cor:aldous-conj} leads to the same conclusion.
	
	As similarly done in Section \ref{sec:HP-sharpness} for the HP, we show that on the homogeneous complete graph the lower bound is attained. Consider $c_{xy}\equiv \frac1n$, $\alpha\equiv \hat \alpha>0$ and $\kappa \in (0,\hat\alpha)$. Then, ${\rm gap}_1(G,\alpha)=\kappa\hat\alpha^{-1}$ and
	\begin{equation}\label{eq:IEM-eigenvalue-2}
		\gamma_{\rm IEM}(G,\alpha,\kappa)\,{\rm gap}_1(G,\alpha,\kappa) = \frac{2\,\kappa\tonde{\hat\alpha-\kappa}}{\hat \alpha^2}\frac{\hat\alpha+\frac1n}{\hat\alpha+1}\fstop
	\end{equation} 
	A direct computation shows that the eigenvalue $\lambda_2>0$ of $-\cL$ associated to the quadratic eigenfunction $f_{a}(\eta)=\sum_{x\in V}\eta_x^2-a$,  $a= \frac1n\frac{\hat\alpha+1}{\hat\alpha+\frac1n}$, equals the expression in \eqref{eq:IEM-eigenvalue-2}, as desired.

\begin{acknowledgement}
	The authors are grateful to Pietro Caputo for valuable discussions. FS thanks the organizers and support staff of the conference \textquotedblleft Intertwining between Probability, Analysis and Statistical Physics\textquotedblright, held at NUS, Singapore, where this project was initiated. While completing this work, MQ and FS were members of GNAMPA, INdAM, and acknowledge financial support through the GNAMPA project \textquotedblleft Redistribution models on networks\textquotedblright.
 SK was supported by Basic Science Research Program through the National Research Foundation of Korea funded by the Ministry of Science and ICT (RS-2025-00518980) and by the Yonsei University Research Fund of 2025 (2025-22-0133).
\end{acknowledgement}


\newcommand{\etalchar}[1]{$^{#1}$}

\end{document}